\documentclass[a4paper,11pt]{article}
\usepackage[latin1]{inputenc}
\usepackage{amsfonts}
\usepackage{amsthm}
\usepackage{amsmath}
\usepackage{amsfonts}
\usepackage{latexsym}
\usepackage{amssymb}
\usepackage{dsfont}
\usepackage[usenames]{color}

\newtheorem{teo}{Theorem}[section]
\newtheorem*{teo*}{Theorem}
\newtheorem{lem}[teo]{Lemma}
\newtheorem{cor}[teo]{Corollary}
\newtheorem{pro}[teo]{Proposition}

\theoremstyle{definition}
\newtheorem{fed}[teo]{Definition}
\theoremstyle{remark}

\newtheorem{rem}[teo]{Remark}

\newtheorem{exas}[teo]{Examples}

\def\coma{\, , \, }

\def\py{\peso{and}}
\newcommand{\peso}[1]{ \quad \text{ #1 } \quad }

\def\n0{n_{ \text{\rm \tiny o}}}

\def\bce{\begin{center}}
\def\ece{\end{center}}

\def\cO{{\mathcal O}}

\def\cJ{{\mathcal J}}

\def\noi{\noindent}
\def\cF{\mathcal F}
\def\cG{\mathcal G}

\def\bm{\left[\begin{array}}
\def\em{\end{array}\right]}
\def\ben{\begin{enumerate}}
\def\een{\end{enumerate}}
\def\bit{\begin{itemize}}
\def\eit{\end{itemize}}
\def\barr{\begin{array}}
\def\earr{\end{array}}

\def\la{\lambda}

\def\N{\mathbb{N}}
\def\R{\mathbb{R}}

\def\cJ{\mathcal{J}}

\def\cC{\mathcal{C}}

\def\cE{\mathcal{E}}

\def\cH{\mathcal{H}}
\def\cK{\mathcal{K}}

\def\cP{\mathcal{P}}

\def\cR{{\cal R}}
\def\cS{{\cal S}}
\def\cT{{\cal T}}
\def\cI{{\cal I}}

\def\cB{{\cal B}}
\def\cN{{\cal N}}
\def\cV{{\cal V}}
\def\cU{{\cal U}}

\def\cL{{\cal L}}

\def\cX{\mathcal{X}}

\def\gli{\mathcal{G}\ell_{\mathcal{J}}}
\def\fS{\mathfrak{S}}

\def\da{^\downarrow}

\DeclareMathOperator{\Tr}{Tr}

\def\beq{\begin{equation}}
\def\eeq{\end{equation}}

\def\pausa{\medskip\noi}

\def\Ax2{\,( S_{E(\cF)^\#_\cV})\hat{}_x }

\newcommand{\PI}[2]{\left\langle #1 , #2 \right\rangle}

\begin{document}
\title{Restricted orbits of closed range operators and equivalences between frames for subspaces}
\author{Eduardo Chiumiento and Pedro Massey}
\date{}
\maketitle

\begin{abstract}
Let $\cH$ be a separable infinite-dimensional complex Hilbert space and let $\cJ$ be a two-sided ideal of the algebra of bounded operators $\cB(\cH)$. The groups $\cG \ell_\cJ$ and $\cU_\cJ$ 
consist of all the invertible operators and unitary operators of the form $I + \cJ$, respectively. We study the actions of these groups on the set of closed range operators. 
First, we find equivalent characterizations of the $\cG \ell_\cJ$-orbits involving the essential codimension. These characterizations can be made more explicit in the case of arithmetic mean closed ideals. Second, we give characterizations of the $\cU_\cJ$-orbits by using recent results on restricted diagonalization. Finally we introduce the notion of $\cJ$-equivalence and $\cJ$-unitary equivalence between frames for subspaces of a Hilbert space, and we apply our abstract results to obtain several results regarding duality and symmetric approximation of $\cJ$-equivalent frames.
\end{abstract}

\bigskip



\tableofcontents

\newpage

\section{Introduction}

Let $\cH$ be a separable infinite-dimensional complex Hilbert space and let $\cS\subset \cH$ be a closed subspace. Briefly, a sequence $\cF=\{f_n\}_{n\geq 1}$ in $\cS$ is called a frame for $\cS$ if every vector $v\in\cS$ can be written as an infinite linear combination (series) of the elements of $\cF$ with coefficients in $\ell_2(\N)$, in such a way that this representation is stable (see Section \ref{sec frames elem teo} for details). In general, a frame $\cF$ allows for redundant representations, a fact that plays an important role in the applications of frame theory.
For every frame $\cF$ for $\cS$ as above, there are two associated bounded linear operators, 
denoted $T_\cF,\,S_\cF\in \cB(\cH)$  that are determined by $T_\cF(e_n)=f_n$ and $S_\cF(v)=\sum_{n\geq 1} \langle v, f_n\rangle f_n$, where $\{e_n\}_{n\geq 1}$ denotes a fixed orthonormal basis of $\cH$ and $v\in\cH$. $T_\cF$ and $S_\cF$ are called the synthesis and frame operator of $\cF$, respectively. It turns out that these operators have closed range and are related by the identity $S_\cF=T_\cF\,T_\cF^*$. 

Many of the fundamental properties of a frame $\cF$ for a closed subspace $\cS$ can be described in terms of properties of its synthesis operator $T_\cF$. This in turn has motivated the use of operator theory in Hilbert spaces to tackle some central problems in frame theory (see \cite{AMRS,EC19,CPS,FPT02} related to the present work). On the other hand, several problems that originated within frame theory have motivated important progresses in operator theory as well (see \cite{Kad02,KL17,L19}).  

In his seminal work on comparisons of frames, Balan \cite{RB99} introduced the notion of equivalent frames for $\cH$. This notion can be described in terms of a natural left action of $\cG\ell(\cH)$, the group of invertible operators acting on $\cH$, on the set of frame operators. Explicitly, two frames $\cF=\{f_n\}_{n\geq 1}$ and
$\cG=\{g_n\}_{n\geq 1}$ for $\cH$ are equivalent if there exists $G\in\cG\ell(\cH)$ such that $G(f_n)=g_n$, $n\geq 1$ or equivalently $GT_\cF=T_\cG$.  A similar type of equivalence was considered by Corach, Pacheco and Stojanoff in \cite{CPS}, but with respect to the right action of $\cG\ell(\cH)$ on the set of frames operators. Moreover, the authors in the latter work endowed the orbits of frames $\cF$ for $\cH$  (i.e. such that $T_\cF$  is an epimorphism) under this right action with a homogeneous space structure. This last fact motivated the study of several geometrical and metric problems corresponding to this right action. It turns out that this perspective corresponds to the study of homogeneous space structures induced by actions of Banach-Lie groups on manifolds defined in operator theory (see \cite{B, Up85}). 
In \cite{FPT02} Frank,  Paulsen and Tiballi considered the symmetric approximation of frames for closed subspaces. In this context, given a frame $\cF=\{f_n\}_{n\geq 1}$ for the subspace $\cS\subset \cH$ the problem is to find the Parseval frames $\cX=\{x_n\}_{n\geq 1}$ for a closed subspace $\cT$ of $\cH$ that are weakly similar to $\cF$, i.e. for which there exists a bounded invertible transformation $L:\cS\rightarrow \cT$ such that $L(f_n)=x_n$, $n\geq 1$, that minimize the expression 
$$\sum_{n\geq 1} \|f_n-x_n\|^2\in[0,\infty]\,.$$ 
Notice that the statement of the problem includes an equivalence-type condition between $\cF$ and $\cX$.

Let $\cB(\cH)$ be the algebra of bounded linear operators on $\cH$. A two-sided ideal $\cJ$ of $\cB(\cH)$ is called an operator ideal. The ideal is said to be proper when $\cJ \neq 0, \, \cB(\cH)$. Motivated by the aforementioned works \cite{RB99,CPS,FPT02}, in the present work we give
a localized version of the equivalence between frames  for closed subspaces, where the localization is with respect to proper operator ideals. Indeed, given a proper operator ideal $\cJ$, we take the restricted invertible group given by $\cG\ell_\cJ:=\cG\ell (\cH) \cap (I + \cJ)$.  We then say that two frames $\cF=\{f_n\}_{n\geq 1}$ and $\cG=\{g_n\}_{n\geq 1}$ for closed subspaces $\cS$ and $\cT$, respectively, are $\cJ$-equivalent if there exists $G \in \cG\ell_\cJ$ such that $G(f_n)=g_n$, $n\geq 1$. 
This setting makes  possible to develop analogs of the symmetric approximation problem for frames for closed subspaces with respect to the operators ideals usually known as symmetrically-normed ideals. 

The analysis of $\cJ$-equivalent frames can be carried out in terms of the synthesis operators of the corresponding frames. This suggests the study of  actions of restricted groups on the set of all closed range operators $\cC \cR$, which would then apply to frames for closed subspaces. We take this last point of view and consider an equivalence relation on $\cC \cR$ in terms of the restricted invertible groups $\cG\ell_\cJ$ induced by proper operator ideals $\cJ$.  In this setting, the equivalence relation is defined by the orbits of the action $G \cdot A=GA$, where $G \in \cG \ell_\cJ$ and $A \in \cC \cR$. 
Also we consider a weaker localized notion of equivalence given by  an action of the product group $\cG\ell_\cJ\times \cG\ell_\cJ$ as follows $(G,K) \cdot A=GAK^{-1}$, where $G,K \in \cG \ell_\cJ$ and $A \in \cC \cR$. We further study the more rigid restricted unitary equivalences induced by the subgroup $\cU_\cJ:=\cU(\cH) \cap (I + \cJ)$, where $\cU(\cH)$ is the full unitary group. These equivalences are defined by taking the restriction of the previous actions to these unitary groups. 
This type of study has already been considered for different operator ideals and closed range operators satisfying some further properties (see \cite{AL08,BPW16,C85,EC10,CMM}).


We first obtain general characterizations of the different orbits induced by the actions of the groups $\cG\ell_\cJ$, $\cG\ell_\cJ\times \cG\ell_\cJ$, $\cU_\cJ$ and $\cU_\cJ\times \cU_\cJ$ on closed range operators, where $\cJ$ is a proper operator ideal. We model some natural problems of frame theory in this abstract framework. Then, we apply our abstract approach in the context of frame theory and obtain several results including the structure of optimal oblique duals and symmetric approximation for $\cJ$-equivalent frames for general  proper operator ideals $\cJ$. 
Since the role of the closed range operators $T_\cF$ and $T_\cF^*$ of a frame $\cF$ for a closed subspace $\cS$ are symmetric, our abstract results can also be interpreted for the right actions of the restricted groups $\cG\ell_\cJ$ and $\cU_\cJ$ on the set of frames for subspaces, corresponding to the point of view considered in \cite{CPS}. We point out that the abstract operator theory approach unveils a rich algebraic and analytic structure of the orbits under consideration. 
Our  tools   include the  index of Fredholm pairs of projections known as essential codimension (see \cite{AS94, ASS94,BDF73,C85, KL17}), and the notion of restricted diagonalization (\cite{BPW16,L19, ECPM1}). We also make use of the theory of operator ideals, arithmetic mean closed operator ideals and their relation with the notion of submajorization (see \cite{Dyk04,KW11}). In particular, we consider the symmetrically-normed  ideals (see \cite{GK60, S79}).

The paper is organized as follows. In Section \ref{sec prelis} we recall some elementary properties of the set of closed range operators, operator ideals, restricted invertible and unitary operators, essential codimension and restricted unitary orbits of partial isometries. In Section \ref{sec3} we introduce the orbits of closed range operators under the action of $\cG\ell_\cJ$ and study some spatial characterizations of the elements of these orbits. Next we characterize when the orbit of a closed range operator under the action of $\cG\ell_\cJ$ contains some special operators, and obtain a result on optimal approximation of operators within this orbit. In Section \ref{sec4} introduce and study the orbits of closed range operators under the actions of the group $\cU_\cJ$ of restricted unitary operators. We include a detailed analysis of closed range operators that admit a block singular value decomposition, which include operators having a finite number of singular values (e.g. partial isometries, normal operators with finite spectrum). In Section \ref{sec res equiv frames} we introduce the different notions of restricted equivalence between frames for closed subspaces of a Hilbert space and apply the results of the previous sections in this setting. In particular, we obtain results related to optimal oblique duals and symmetric approximation of frames with respect to symmetrically-normed ideals.

\section{Preliminaries}\label{sec prelis}
 Let $\cB(\cH)$ be the algebra of bounded linear operators acting on $\cH$, and let $\cB(\cH)^+$ be the cone of positive semidefinite operators. Given $A \in \cB(\cH)$, we write $R(A)$ and $N(A)$ for the range and nullspace of $A$, respectively, and let $\sigma(A)$ be the spectrum of $A$. The orthogonal projection onto 
a closed subspace $\cS \subseteq \cH$ is denoted by $P_\cS$.

\medskip

\noi \textit{Closed range operators.}  The set of all \textit{closed range operators} on $\cH$ is given by 
$$
\cC\cR=\{ A \in \cB(\cH)\, : \, R(A) \text{ is a closed subspace} \}.
$$
For $A \in \cB(\cH)$, $A\neq 0$, the \textit{reduced minimum modulus}  is  $\gamma(A)=\inf \{ \|Af\| :  f \in N(A)^\perp, \, \|f\|=1\}$. It is well known that $A \in \cC \cR$ if and only if $\gamma(A)>0$. It can be shown that $\gamma(A)=\min_{\lambda \in \sigma(|A|)\setminus \{ 0\}} \lambda$, where $|A|=(A^*A)^{1/2}$ is the operator modulus. From the latter characterization, it follows that  $\gamma(|A|)^2=\gamma(A)^2=\gamma(AA^*)=\gamma(A^*A)=\gamma(A^*)^2=\gamma(|A^*|)^2$. In particular, $\cC\cR\subset \cB(\cH)$ is closed under taking operator adjoint and  modulus. 
The positive part of $\cC \cR$  is denoted by $\cC \cR^+:= \cB(\cH)^+ \cap \cC \cR$.

 For an operator $A \in \cC\cR$, its Moore-Penrose inverse $A^\dagger$ is the bounded linear operator
uniquely determined by the four conditions $AA^\dagger A=A$, $A^\dagger AA^\dagger =A^\dagger$,  $A A^\dagger=P_{R(A)}$ and 
$A^\dagger A=P_{N(A)^\perp}$. For $A\in \cC\cR$,  $\gamma(A)=\|A^\dagger\|^{-1}$, where $\| \, \cdot \, \|$ is the operator norm. 
The following useful identity was proved in  \cite{W73, S77} for matrices. It can be generalized to closed range operators following the same proof: for $A,B \in \cC\cR$, 
\begin{equation}\label{stewart identity}
A^\dagger - B^\dagger=-A^\dagger(A-B)B^\dagger + A^\dagger (A^* )^\dagger (A^* - B^*)(I - BB^\dagger) + (I-A^\dagger A)(A^* - B^*) (B^*)^\dagger  B.
\end{equation}  
Recall that an operator $V \in \cB(\cH)$ is a partial isometry  if $VV^*V=V$, or equivalently, $VV^*$ is an orthogonal projection. This is also equivalent to say that $\|Vf\|=\|f\|$, for all $f \in N(V)^\perp$.  We use the notation:
$$
\cP\cI=\{  V \in \cB(\cH) :   V  \text{ is a partial isometry }\}.
$$
It follows easily that $\cP\cI \subset \cC\cR$,  $VV^*=P_{R(V)}$ (\textit{final projection}), and $V^*V$ is also a projection satisfying $V^*V=P_{N(V)^\perp}$(\textit{initial projection}).

We use the notation $A=V_A|A|$ for the polar decomposition of an operator $A \in \cB(\cH)$, where 
$|A|\in\cB(\cH)^+$  and $V_A$ is the partial isometry uniquely determined by the condition $N(V_A)=N(A)$.
 We recall here the following characterization: if $A=VC$ for a positive operator $C$ and  a partial isometry $V$ such that $R(V^*)=\overline{R(C)}$, then it must be $V=V_A$ and $C=|A|$.

\medskip

\noi \textit{Operator ideals and restricted groups.} 
An \textit{operator ideal} is a two-sided ideal $\cJ$ of $\cB(\cH)$. The ideal is \textit{proper} when 
$\{0\}\neq \cJ\neq \cB(\cH)$. Operator ideals are closed under the operator adjoint ($A^* \in \cJ$ whenever $A \in \cJ$). Another useful property is the following:  $\cF \subseteq \cJ  \subseteq \cK $, for any proper operator ideal $\cJ$, where $\cF=\cF(\cH)$ and $\cK=\cK(\cH)$ denote the ideals of finite-rank operators and compact operators on $\cH$, respectively. 

Apart from the finite-rank and compact operators, the $p$-Schatten ideals $\fS_p=\fS_p(\cH)$, for $0<p \leq \infty$, are operator ideals. Recall that for $p>0$  these operator ideals are defined as
$$
\fS_p:=\{  \,   A \in \cK    \, : \, \|A \|_p:=\Tr(|A|^p)^{1/p} < \infty   \, \}.
$$
For $p=\infty$,  set $\fS_\infty=\cK$ and $\| \, \cdot \, \|_\infty=\| \, \cdot \, \|$ is the usual operator norm. Other examples of operator ideals can be constructed as follows. Let $c_{00}=c_{00}(\N)$ be the real vector space consisting of all sequences with a finite number of nonzero terms.  A symmetric norming function is a norm $\Phi: c_{00} \to \R_{\geq 0}$ satisfying the following properties: $\Phi(1,0,0,\ldots)=1$ and $\Phi(a_1, a_2 , \ldots , a_n, 0, 0 , \ldots)=\Phi(|a_{\sigma(1)}|, |a_{\sigma(2)}| , \ldots , |a_{\sigma(n)}|, 0, 0 , \ldots)$, where  $\sigma$  is any permutation of the integers $1,2,\ldots, n$ and $n\geq 1$. 
Recall that the singular values of an operator $A \in \cK$ are the eigenvalues  of $|A|$. We denote by $\{ s_n(A)\}_{n \geq 1}$ the sequence of the singular values of $A$ arranged in non-increasing order and counting multiplicities. 
 Then, using the singular values of  $A\in\cK$ one can define: 
\[  \| A\|_{\Phi}:= \sup_{k \geq 1} \Phi (s_1(A), s_2(A), \ldots, s_k(A), 0, 0 , \ldots) \in [0,\infty].   \] 
 It turns out that  $\fS_\Phi=\fS_{\Phi}(\cH):= \{  \,   A \in \cK    \, : \, \|A \|_{\Phi} < \infty   \, \}$ are operator ideals, which are usually known as \textit{symmetrically-normed ideals}. Notice that this type of operator ideals are proper. The $p$-Schatten ideals for $1\leq p \leq \infty$ are particular cases of the symmetrically-normed ideals associated with the symmetric norming functions $\Phi_p=\| \, \cdot \, \|_{\ell^p}$. For every symmetric norming function $\Phi$, we have the following estimates involving the operator norm and the ideal norm:
\begin{align}
  \| A \| &\leq \|A\|_\Phi , \, \, \, \, \,\, A \in \fS_\Phi , \label{comp op norm} \\
 \| ABC\|_\Phi &\leq \|A\| \|B\|_\Phi \|C\|, \, \, \, \, \, \, B \in \fS_\Phi, \,  A,C \in \cB(\cH).    \label{estimate sym norm ideal} 
 \end{align}
In particular, the norm $\| \, \cdot \, \|_\Phi$ is sub-multiplicative, i.e. $\|B_1 B_2\|_\Phi \leq \|B_1\|_\Phi \|B_2\|_\Phi$, $B_1, B_2 \in \fS_\Phi$. 
The proofs of these facts and other examples of symmetrically-normed ideals for specific symmetric norming functions $\Phi$ can be found in \cite{GK60}.
 For further examples of operator ideals such as Lorentz ideals, Marcinkiewicz ideals and Orlicz ideals we refer to \cite{Dyk04}.

Given two compact operators $A,\,B\in \cK$ we say that $s(A)=\{s_n(A)\}_{n\geq 1}$ {\it submajorizes} $s(B)=\{s_n(B)\}_{n\geq 1}$, denoted $s(B)\prec_w s(A)$, if for every $n\geq 1$ we have that 
$$
s_1(B) + \ldots + s_n(B) \leq s_1(A) + \ldots + s_n(A)\,.
$$ 
An operator ideal $\cJ$ is called \textit{arithmetic mean closed} if the following property holds: if  $B \in \cB(\cH)$ and $A \in \cJ$ are such that $s(B)\prec_w s(A)$ then $B\in\cJ$.
The symmetrically-normed ideals $\fS_\Phi$ are arithmetic mean closed (see \cite[p.82]{GK60}). Moreover, in this case the norm $\|\cdot\|_\Phi$ is \textit{Schur-convex} in the sense that if $A,\,B\in\cJ$ are such that $s(B)\prec_w s(A)$ then $\|B\|_\Phi\leq \|A\|_\Phi$. We further say $\|\cdot\|_\Phi$ is \textit{strictly Schur-convex} (equivalently,  the symmetric norming function $\Phi$ is strictly Schur-convex) if whenever $A,\,B\in\cJ$ are such that $s(B)\prec_w s(A)$ and $\|B\|_\Phi=\|A\|_\Phi$, then we have $s(A)=s(B)$.
Notice that for $1 \leq p \leq \infty$, $\fS_p$ is arithmetic mean closed; while $\|\cdot\|_p$ is strictly Schur-convex for $1<p<\infty$. Elementary examples of non arithmetic mean closed operator ideals are given by $\cF$ and $\fS_p$ for $0<p<1$.

Denote by $\cG\ell(\cH)$ and $\cU(\cH)$ the invertible group and the unitary group of the Hilbert space $\cH$, respectively. 
Given an operator ideal $\cJ$, we consider the \textit{restricted invertible group}, i.e.
$$
\gli:=\{ G \in \cG\ell(\cH): G- I \in \cJ \},
$$
and the \textit{restricted unitary group}, i.e.
$$
\cU_\cJ:=\{ U \in \cU(\cH): U- I \in \cJ \}.
$$
These types of groups are nontrivial and strictly smaller subgroups  of the invertible group and  the full unitary group whenever $\cJ$ is a proper operator ideal.
They can be studied  from a geometrical 
viewpoint for many operator ideals. For instance, if the operator ideal $\cJ$ admits a complete norm $\| \, \cdot \,\|_\cJ$ stronger than the operator 
norm, then $\gli$ and $\cU_\cJ$ turn out to be complex and real Banach-Lie groups, respectively,  in the topology defined by $d_\cJ(G_1,G_2)=\| G_1 - G_2 \|_\cJ$, for $G_1, G_2 \in \gli$, or $G_1, G_2 \in \cU_\cJ$   (see \cite[Prop. 9.28]{B}).  These conditions are satisfied by the $p$-Schatten ideals $\fS_p$ for $1\leq p \leq \infty$, and more generally, 
by the symmetrically-normed ideals $\fS_\Phi$.

\medskip

\noi \textit{Essential codimension.} Let $P,Q \in \cB(\cH)$ be two (orthogonal) projections such that  the operator $QP|_{R(P)}:R(P)\to R(Q)$ is Fredholm. In this case,  $(P,Q)$ is known as a \textit{Fredholm pair} and the index of this Fredholm operator  
\begin{align}
[P:Q]& :=\mathrm{Ind}(QP|_{R(P)}:R(P)\to R(Q)) \label{index pair}\\
& = \dim(N(Q)\cap R(P)) - \dim(R(Q)\cap N(P)) \nonumber
\end{align}
is called the \textit{essential codimension}. The notion of essential codimension was introduced in \cite{BDF73}, and studied in 
several works (see \cite{AS94, ASS94}). Notice that when $P- Q \in \cK$, then $(P,Q)$ is a Fredholm pair.

\begin{rem}\label{prop ess cod}
We list here some helpful properties 
of the essential codimension.
\begin{itemize}
\item[i)] $[P:Q]=-[Q:P]$.
\item[ii)] Let $(P_i,Q_i)$, $i=1,2$, be two Fredholm pairs.  If  $P_1P_2=0$ and $Q_1Q_2=0$, then $(P_1+ P_2,Q_1+Q_2)$ is a Fredholm pair and
$
[P_1 + P_2:Q_1 + Q_2]=[P_1:Q_1] + [P_2:Q_2].
$ 
\item[iii)] If $(P,Q)$ and $(Q,R)$ are Fredholm pairs, 
and either $Q-R$ or $P-Q$ are compact, then $(P,R)$ is a Fredholm pair and 
$[P:R]=[P:Q] + [Q:R]$.
\end{itemize}
\end{rem} 

\begin{rem}\label{unitary orb partial iso and proj}
We recall some previous results on the action of the $\cJ$-restricted unitary groups using the notion of essential codimension.
\begin{enumerate}
\item[i)] Let $V_1, V_2\in \cP\cI$ and let $\cJ$ be a proper operator ideal.  
Then there exists a unitary operator $U \in \cU_\cJ$ such that $UV_1=V_2$ if and only if $N(V_1)=N(V_2)$ and $V_1-V_2 \in \cJ$ (see \cite{EC10, ECPM1}).
\item[ii)]  Let $V_1, V_2\in \cP\cI$ and let $\cJ$ be a proper operator ideal. The following assertions are equivalent (see \cite{EC10}):
\begin{enumerate}
\item[a)] There exist $U, W \in \cU_\cJ$ such that $UV_1W^*=V_2$;   
\item[b)] $V_1 - V_2 \in \cJ$ and $[V_1 V_1^*: V_2 V_2^*]=0$; 
\item[c)] $V_1 - V_2 \in \cJ$ and $[V_1^* V_1: V_2^* V_2]=0$.
\end{enumerate}
\item[iii)] Let $P$, $Q$ be orthogonal projections and let $\cJ$ be a proper operator ideal.
Then there is a unitary operator $U \in \cU_\cJ$
such that $Q=UPU^*$ if and only if $P-Q \in \cJ$ and $[P:Q] = 0$ (see \cite{KL17}).
\end{enumerate}
\end{rem}

\section{Orbits of the restricted invertible group}\label{sec3}

In this section we study orbits of closed range operators under two actions defined by the group $\gli$.
The first action consists of left multiplication $\gli \times \cC \cR \to \cC \cR$, $G \cdot A=GA$. For an operator $A \in \cC\cR$, 
 we consider the corresponding orbit, i.e. 
$$
\cL\cO_\cJ(A)=\{ GA : G \in \gli \}.
$$
The other action is given by the group $\gli \times \gli$, and it consists of both left and right multiplication. More precisely, it is defined by the map $(\gli \times \gli )\times \cC \cR \to \cC\cR$,  $(G,K)\cdot A=GAK^{-1}$, where $A \in \cC\cR$ and $G,K \in \gli$. For any $A \in \cC\cR$, the orbit given by this action is
$$
\cO_\cJ(A)=\{ GAK^{-1} : G, K \in \gli \}.
$$ 

\subsection{Spatial characterizations}\label{spatial ch both orb}

In what follows we will present several characterizations of the two orbits defined above.

\begin{lem}\label{ab y moore penrose}
Let $\cJ$ be an operator ideal. Let $A,B \in \cC\cR$ be such that $A-B \in \cJ$. Then $A^\dagger - B^\dagger \in \cJ$,    $P_{R(A)}-P_{R(B)} \in \cJ$ and $P_{N(A)}-P_{N(B)} \in \cJ$. \qed
\end{lem}
\begin{proof}
Note that $A^\dagger - B^\dagger \in \cJ$ follows immediately from  Eq. \eqref{stewart identity}, because we know that $A- B \in \cJ$ and $A^*-B^* \in \cJ$. For the other statement, observe that 
one can write $P_{R(A)}-P_{R(B)}=AA^\dagger - BB^\dagger=A(A^\dagger - B^\dagger) -  (B -A)B^\dagger \in \cJ$.   Similarly, we have $P_{N(A)}-P_{N(B)} \in \cJ$, since $I-P_{N(A)}=A^\dagger A$ and
$I-P_{N(B)}=B^\dagger B$.
\end{proof}

Corach, Maestripieri and Mbekhta established a relation between orbits of closed range operators and unitary orbits of the (reverse) partial isometries of the polar decompositions of closed range operators (see \cite[Prop. 5.3]{CMM}). This motivates our following results on the relation of orbits of closed range operators and unitary orbits of their partial isometries in the polar decomposition, in our setting of restricted orbits. 
Next, we recall a result by van Hemmen and Ando, and we then apply it to give characterizations of the orbits of a closed range operator. 
\begin{pro}[\cite{AvH}]\label{Ando}
If $A,B \geq 0$ and $A^{1/2} + B^{1/2} \geq \mu I \geq 0$, then for any symmetric norming function $\Phi$
$$
\|A-B\|_\Phi \geq \mu \| A^{1/2} - B^{1/2}\|_\Phi \, .
$$
\end{pro}

\begin{lem}\label{lem sobre raices en ideales}
Let $\cJ$ be an arithmetic mean closed operator ideal and let $D,\,E\in\cC\cR$ be positive operators such that $N(D)=N(E)$ (or equivalently $R(D)=R(E)$) and $D^2 -E^2\in\cJ$. Then $D-E\in\cJ$. 
\end{lem}
\begin{proof}
Notice that $R(D^2 )=N(D^2)^\perp=N(D)^\perp$ and that $D^2 \in \cC\cR$. Similarly,
$R(E^2 )=N(E^2)^\perp=N(E)^\perp$ and that $E^2 \in \cC\cR$. In particular, $R(D)=R(E)$.
 We are going to apply Proposition \ref{Ando} in the Hilbert space 
$\cH_0:=R(D)=R(E)$ with the operators $A:=E^2|_{\cH_0}\in \cB(\cH_0)$ and $B:=E^2|_{\cH_0}\in \cB(\cH_0)$, which by the remarks in the previous paragraph are invertible on $\cH_0$. We thus get $A^{1/2} + B^{1/2} \geq \mu I_{\cH_0}$ for some $\mu>0$. Consider the operator ideal on $\cH_0$ defined by $\cJ_0:=\{ P_{\cH_0}T|_{\cH_0} : T \in \cJ \}$, which is also arithmetic mean closed. Notice that by construction
$A-B\in\cJ_0$.   Taking the symmetric norming functions $\Phi_n(x_1, \ldots , x_n ,x_{n+1}, \ldots)=|x_1| +  \ldots + |x_n|$, $n \geq 1$, then by Proposition \ref{Ando} we find that 
$s_1(A-B) +  \ldots + s_n(A-B) \geq \mu \ (s_1(A^{1/2}-B^{1/2}) +  \ldots + s_n(A^{1/2}-B^{1/2}))$, for all $n \geq 1$; that is, $s(\mu (A^{1/2}-B^{1/2}))\prec_w s(A-B)$. This implies that $A^{1/2}-B^{1/2}\in \cJ_0$ because $\cJ_0$ is arithmetic mean closed. Using again that $\cH_0=R(D)=R(E)$,  we obtain $D - E \in \cJ$.  
\end{proof}


\begin{pro}\label{polar dec y a menos b}
Let $A,B \in\cC\cR$ satisfying $[P_{N(A)}:P_{N(B)}]=0$ and let $\cJ$ be an arithmetic mean closed operator ideal. The following conditions are equivalent:
\begin{itemize}
\item[i)] $A-B \in \cJ$;
\item[ii)] $|A|-|B| \in \cJ$ and $V_A - V_B \in \cJ$. \qed
\end{itemize} 
\end{pro}
\begin{proof}
$i) \to ii)$. If $A-B \in \cJ$, then $A^*-B^* \in \cJ$, 
so Lemma \ref{ab y moore penrose} shows that $P_{R(A^*)}-P_{R(B^*)}\in\cJ$. Thus,
$P_{N(A)}-P_{N(B)}=P_{R(B^*)}-P_{R(A^*)}\in\cJ$. Since $[P_{N(A)}:P_{N(B)}]=0$ then, by item $iii)$ in Remark  \ref{unitary orb partial iso and proj}, we see that there exists $U\in\cU_\cJ$ such that $N(B)=U\,N(A)$, and in this case, we get that $N(AU^*)=N(B)$. Notice that 
$AU^*-B=A(U-I)^*+ A-B \in \cJ$. The previous remarks show that if we let $C=AU^*$, then $C\in\cC\cR$, $N(|C|)=N(C)=N(B)=N(|B|)$ and $C-B\in\cJ$ so that $|C|^2 -|B|^2 =C^*(C-B)+(C^*-B^*)B\in\cJ$. By Lemma \ref{lem sobre raices en ideales} we conclude that $|C|-|B|\in\cJ$. Now we can use Lemma \ref{ab y moore penrose} to deduce that $V_C-V_B=(|C|^\dagger - |B|^\dagger)C +  |B|^\dagger(C-B) \in \cJ$.
Since 
$|A|=U^*|C|U$, then $|A|-|C|\in\cJ$. On the other hand, the identity $C=AU^*=(V_AU^*U|A|)U^*=V_AU^* |C|$ together with $R(|C|)=UR(|A|)=N(V_AU^*)^\perp$ imply that $V_C=V_AU^*$ by the characterization of the partial isometry and positive operator in the polar decomposition. Hence $V_A-V_C\in\cJ$. These last identities show that $|A|-|B|=(|A|-|C|)+(|C|-|B|)\in\cJ$ and $V_A-V_B=(V_A-V_C)+(V_C-V_B)\in \cJ$.

\medskip

\noi $ii) \to i)$. This follows from the fact that $A-B=(V_A-V_B)|A|+V_B(|A|-|B|)$. 
\end{proof}

We first consider the orbit of a closed range operator $A$ given by the left by multiplication of operators in  $\cG\ell_\cJ$.

\begin{teo}\label{para frames J equiv} 
Let $A,B \in\cC\cR$ and $\cJ$ be a proper operator ideal. Then the following conditions are equivalent:
\begin{enumerate}
\item[i)] There exists $G \in \gli$ such that $GA=B$; 
\item[ii)] $A-B \in \cJ$ and $N(A)=N(B)$.
\end{enumerate}
If any of the conditions above hold, then $[P_{R(A)}:P_{R(B)}]=0$. 

Moreover, if we assume further that $\cJ$ is arithmetic mean closed, then the conditions above are also equivalent to
\begin{enumerate}
\item[iii)] $|A|-|B|\in\cJ$, $V_A-V_B\in\cJ$ and $V_A^*V_A=V_B^*V_B$, where $A=V_A|A|$ and $B=V_B |B|$ are the polar decompositions of $A$ and $B$, respectively.
\end{enumerate}
\end{teo}
\begin{proof}
\noi $i) \to ii)$. Assume that $B=GA$, where $G \in \gli$. Then it is clear that $N(A)=N(B)$. On the other hand, 
$$
A-B = A-GA=(I-G)A\in\cJ,
$$ 
since $I-G\in\cJ$ by assumption.

\medskip

\noi $ii) \to i)$. Assume that  $A,\,B\in\cC\cR$ are such that $N(A)=N(B)$ and $A-B\in\cJ$. 
Notice that  $P_{R(A^*)}$ and $P_{R(B^*)}$ coincide with the orthogonal projection onto $N(A)^\perp=N(B)^\perp$. Hence, if we let 
$D=BA^\dagger$, then $DA=BA^\dagger A=B P_{N(A)^\perp}=B$. In this case, 
$$
D-P_{R(A)}=BA^\dagger-AA^\dagger=(B-A)A^\dagger\in\cJ\,.
$$
If we let $K\in \cJ$ be such that 
$D=P_{R(A)}+K$, then
$$
D^*D=(P_{R(A)}+K)^*(P_{R(A)}+K)=P_{R(A)}+K^*P_{R(A)}+P_{R(A)}K+K^*K,
$$
which shows that $|D|^2-P_{R(A)}\in\cJ$. Since $A, B \in \cC\cR$ and $N(A)=N(B)$, the operator $D|_{R(A)}:R(A) \to R(B)$ is invertible, and consequently, the operator $|D||_{R(A)}:R(A)\rightarrow R(A)$ is also invertible. Furthermore, $R(D)=R(B)$ and $R(A)=R(D^*)$.
In particular, we see that
$$
|D|-P_{R(A)}=(|D|+P_{R(A)})^\dagger(|D|^2-P_{R(A)})\in\cJ, 
$$
which implies $|D|^\dagger-P_{R(A)}\in\cJ$ by Lemma \ref{ab y moore penrose}.
Consider the polar decomposition $D=U\,|D|$, where $U$ is a partial isometry with initial projection $P_{R(A)}$ and final projection $P_{R(B)}$. In this case,
$$
U-P_{R(A)}=(D-P_{R(A)})|D|^\dagger+(|D|^\dagger-P_{R(A)})\in\cJ\,.
$$
Since $N(U)=N(D)=N(P_{R(A)})$, then
by Remark \ref{unitary orb partial iso and proj} $i)$ 
we conclude that there exists $Z\in\cU_\cJ$ such that 
$Z P_{R(A)}=U$. In particular, we get $ZP_{R(A)}Z^*=UU^*=P_{R(B)}$, so that $[\,P_{R(A)}\,:\,P_{R(B)}\,]=0$ by Remark item \ref{unitary orb partial iso and proj} $iii)$. 

We now set $Y=|D|+(I-P_{R(A)})$ and notice that $Y\in\cG\ell_\cJ$. Indeed,
$Y$ is invertible (recall that $N(|D|)=R(A)^\perp$ and $|D||_{R(A)}:R(A)\rightarrow R(A)$ is an invertible operator) and
$I-Y=|D|-P_{R(A)}\in\cJ$. Finally, notice that 
$
ZYA=Z|D|A=ZP_{R(A)}|D|A=U|D|A=DA=B$ and $ZY \in\cG\ell_\cJ$.

 Assume now that $\cJ$ is an arithmetic mean closed ideal. Notice that the equivalence between items $ii)$ and $iii)$ follows from Proposition \ref{polar dec y a menos b} and the fact that $V_A^*V_A=I-P_{N(A)}$ and similarly for $B$.
\end{proof}

The following is a direct consequence of Theorem \ref{para frames J equiv}.
\begin{cor}\label{coro caract orb1}
Let $A\in\cC\cR$ and $\cJ$ be a proper operator ideal. Then, 
$$
\cL\cO_\cJ(A)=\{B\in\cC\cR\ : \ A-B\in\cJ \text{ and } N(A)=N(B) \}\,.
$$
Moreover, if we assume further that $\cJ$ is arithmetic mean closed, then
$$
\cL\cO_\cJ(A)=\{B\in\cC\cR\ : \ |A|-|B|\in\cJ \coma V_A-V_B\in\cJ \,\text{ and } \,V_A^*V_A=V_B^*V_B\}\,.
$$
\end{cor}

Now we can derive a characterization of the orbit defined by multiplication on both sides.

\begin{teo}\label{bi-orbit charactalt}
Let $A,B \in\cC\cR$ and $\cJ$ be a proper operator ideal. The following conditions are equivalent:
\begin{enumerate}
\item[i)] $B=GAK^{-1}$, for some $G,K \in \gli$;
\item[ii)] $A-B \in \cJ$ and $[P_{N(A)}:P_{N(B)}]=0$;
\item[iii)] $A-B \in \cJ$ and $P_{N(B)}=UP_{N(A)}U^*$, for some $U\in \cU_\cJ$.
\end{enumerate}
Moreover, if we assume further that $\cJ$ is arithmetic mean closed, then the conditions above are also equivalent to
\begin{enumerate}
\item[iv)] $|A|-|B|\in\cJ$, $V_A-V_B\in\cJ$ and $[V_A^*V_A: V_B^*V_B]=0$.
\end{enumerate}
\end{teo}
\begin{proof}
\noi $i) \rightarrow ii)$. The fact that  $A-B\in \cJ$ is a direct computation. By Theorem \ref{para frames J equiv} and the identity $R(A^* G^*)=R(A^*)$  we get that $[P_{N(A)}\, :\,P_{N(B)}]=-[\, P_{R(A^*)}\,:\,P_{R(B^*)}\,]=0$.

\medskip

\noi $ii) \rightarrow iii)$. This is a consequence of Remark \ref{unitary orb partial iso and proj} $iii)$ and Lemma \ref{ab y moore penrose}.

\medskip

\noi $iii) \rightarrow i)$. Let $A'=AU^*$ and notice that $A'-B=A'-A+A-B=A(U^*-I)+A-B\in\cJ$. Moreover, $N(A')=U (N(A))=N(B)$. By Theorem \ref{para frames J equiv} we conclude that there exists $G\in\cG\ell_\cJ$ such that $GA'=B$. In this case
$B=GA'=GAU^*$, and item $i)$ holds with $G\in\cG\ell_\cJ$ and $K=U\in\cU_\cJ\subset\cG\ell_\cJ$.

 Assume now that $\cJ$ is an arithmetic mean closed ideal. Notice that the equivalence between items $ii)$ and $iv)$ follows from Proposition \ref{polar dec y a menos b}, the fact that $V_A^*V_A=I-P_{N(A)}$ and similarly for $B$, and item $ii)$ in Remark \ref{unitary orb partial iso and proj}.
\end{proof}

\begin{cor}\label{coro caract orb2}
Let $A\in\cC\cR$ and $\cJ$ be a proper operator ideal. Then, 
$$
\cO_\cJ(A)=\{B\in\cC\cR\ : \  A-B\in\cJ \text{ and } [P_{N(A)}\,:\,P_{N(B)}]=0\, \}\,.
$$
Moreover, if we assume further that $\cJ$ is arithmetic mean closed, then
$$
\cO_\cJ(A)=\{B\in\cC\cR\ : \ |A|-|B|\in\cJ \coma V_A-V_B\in\cJ \,\text{ and } \,[V_A^*V_A\,:\,V_B^*V_B]=0\,\}\,.
$$
\end{cor}

\subsection{Special representatives and optimal approximations}\label{sec espec rep}

Let $\cJ$ be a proper operator ideal and let $A \in \cC \cR$. Observe that two orbits must be disjoint or equal. 
That is, $B\in\cL\cO_\cJ(A)$ (resp. $B\in \cO_\cJ(A)$) if and only if  $\cL\cO_\cJ(A)=\cL\cO_\cJ(B)$ (resp. $\cO_\cJ(A)=\cO_\cJ(B)$). Our next result characterizes when there are partial isometries $V$ or orthogonal projections $P$ that are representatives of $\cL\cO_\cJ(A)$ and $\cO_\cJ(A)$.

\begin{teo}\label{teo de ops para framesx28}
Let $\cJ$ be a proper operator ideal and let $A \in \cC \cR$ with polar decomposition $A=V_A|A|$. Then the following conditions are equivalent: 
\begin{itemize}
\item[i)] $AA^*- P_{R(A)} \in \cJ$ (equivalently $(AA^*)^\dagger - P_{R(A)} \in \cJ$);
\item[ii)] $(A^*)^\dagger \in \cL \cO_\cJ(A)$;
\item [iii)] There exists $B\in\cL\cO_\cJ(A)$ such that $AB^*|_{R(A)}=I_{R(A)}$, $BA^*|_{R(B)}=I_{R(B)}$;
\item[iv)] $|A^*|-P_{R(A)}\in \cJ$;
\item [v)] $V_A\in \cL \cO_\cJ(A)$;
\item [vi)] There is a partial isometry $V \in \cL \cO_\cJ(A)$; 
\item[vii)] There is a partial isometry $V \in \cO_\cJ(A)$;
\item [viii)] There is a partial isometry $V \in \cO_\cJ(A^*)$; 
\item [ix)] There is an orthogonal projection $P\in \cL\cO_\cJ(A)$; in this case $P=P_{R(A^*)}$. 
\end{itemize}
\end{teo}
\begin{proof}
$i) \rightarrow  ii).$ If $AA^*- P_{R(A)} \in \cJ$, then by Lemma \ref{ab y moore penrose} $(AA^*)^\dagger - P_{R(A)} \in \cJ$; hence $(A^*)^\dagger = (AA^*)^\dagger A =((AA^*)^\dagger+P_{R(A)^\perp}) A$, where $(AA^*)^\dagger+P_{R(A)^\perp}\in\cG\ell$ is such that $(AA^*)^\dagger+P_{R(A)^\perp}-I=(AA^*)^\dagger - P_{R(A)} \in \cJ$.

\medskip

\noi
$ii) \leftrightarrow  iii).$ The forward implication is clear; conversely, assume that $B=GA$ for $G\in\cG\ell_\cJ$ and such that $AB^*|_{R(A)}=I_{R(A)}$, $BA^*|_{R(B)}=I_{R(B)}$. Then, $AB^*A=A$ and hence $A^*(GA) A^*=A^*$; thus, $P_{R(A)}GP_{R(A)}=(AA^*)^\dagger$. If we let $H=P_{R(A)}GP_{R(A)}+P_{R(A)^\perp}$, then $H\in\cG\ell(\cH)$, $H-I=P_{R(A)}GP_{R(A)}-P_{R(A)}=P_{R(A)}(G-I)P_{R(A)}\in\cJ$ and is such that $HA=(AA^*)^\dagger A=(A^*)^\dagger$.

\medskip

\noi
$ii) \rightarrow  iv).$ If $(A^*)^\dagger \in \cL \cO_\cJ(A)$, then in particular, $(A^*)^\dagger-A\in\cJ$. Thus,
$P_{R(A)}-|A^*|^2=P_{R(A)}-AA^*=((A^*)^\dagger -A)A^*\in\cJ$; hence, $P_{R(A)}-|A^*|=(P_{R(A)}+|A^*|)^\dagger (P_{R(A)}-|A^*|^2)\in\cJ$. 

\medskip

\noi
$iv) \rightarrow  v).$
 Lemma \ref{ab y moore penrose} shows that 
$P_{R(A)}-|A^*|^\dagger\in \cJ$. Now, using the identity $A^*=V_A^*|A^*|$ we see that $A=|A^*|V_A$ and hence $V_A=|A^*|^\dagger A=(|A^*|^\dagger + P_{R(A)^\perp})A$, where $|A^*|^\dagger + P_{R(A)^\perp}\in \cG\ell$ is such that $|A^*|^\dagger + P_{R(A)^\perp}-I=|A^*|^\dagger - P_{R(A)}\in\cJ$.

\medskip

\noi
$v) \rightarrow  vi)$ and $vi) \rightarrow  vii)$ are straightforward.

\medskip 

\noi $vi)\rightarrow  i)$.  Suppose that $V=GA$ is a partial isometry for some $G \in \cG\ell_\cJ$. Then, note that $A^*G^*GA=V^*V=P_{N(A)^\perp}$, or equivalently, 
$P_{R(A)}G^*GP_{R(A)}=(A^*)^\dagger A^\dagger=(AA^*)^\dagger$. Since $P_{R(A)}G^*GP_{R(A)} - P_{R(A)} \in \cJ$, it follows that $(AA^*)^\dagger - P_{R(A)} \in \cJ$. Hence, by Lemma \ref{ab y moore penrose} we have that $AA^*-P_{R(A)}\in\cJ$.

\medskip

\noi $vii)\rightarrow  i).$  Assume that $GAK^{-1}=V$ is a partial isometry, for $G,\,K\in \cG\ell_\cJ$.
Then, in this case, $V\in\cL\cO_\cJ(AK^{-1})$. Using the implication $vi) \rightarrow i)$ that we have already proved, we now see that $AK^{-1} (AK^{-1})^* - P_{R(AK^{-1})}=AK^{-1} (AK^{-1})^*  -P_{R(A)}\in\cJ$. Since
$K^{-1}\in\cG\ell_\cJ$ we get that $|(K^{-1})^*|^2  - I\in\cJ$ and then $ AA^*-AK^{-1} (AK^{-1})^*=A(I-|(K^{-1})^*|^2)A^*\in\cJ$. 
 The previous facts show that 
$$
AA^* - P_{R(A)}= AA^*- AK^{-1} (AK^{-1})^* + AK^{-1} (AK^{-1})^* -P_{R(A)}\in\cJ\,.
$$
Notice that we have that  $i)-vii)$ are all equivalent.

\noi $vii)\leftrightarrow  viii)$. Since $\cG\ell_\cJ$ is closed under the adjoint operation, we see that 
the partial isometry $V\in \cO_\cJ(A)$ if and only if the $V^*\in \cO_\cJ(A^*)$.

\medskip

\noi $viii)\rightarrow  ix)$. 
 Notice that the equivalence of item $i)-viii)$ allow us to replace the role of $A$ by $A^*$ in items $i)-vii)$ above. In particular, replacing $A$ by $A^*$ in $iv)$, we get that $ |A|-P_{R(A^*)}\in\cJ$. Since $P\in \cL\cO_\cJ(A)$, then $A-P\in\cJ$ and $N(P)=N(A)$, and thus, $P=P_{R(A^*)}$. Hence, 
$V_A-P_{R(A^*)}=V_A-A+A-P_{R(A^*)}=V_A(P_{R(A^*)}-|A|)+(A-P_{R(A^*)})\in\cJ$ and $N(V_A)=N(P_{R(A^*)})$.
By Theorem \ref{para frames J equiv} we see that $P_{R(A^*)}\in\cL\cO_\cJ(V_A)=\cL\cO_\cJ(A)$, where the last equality follows from the fact that $viii)\rightarrow v)$.

\medskip

\noi $ix)\rightarrow vi)$ is clear, since $P$ is a partial isometry. 
\end{proof}

\begin{cor}
Let $\cJ$ be a proper operator ideal and let $A \in \cC \cR$ with polar decomposition $A=V_A|A|$. Then the following conditions are equivalent: 
\begin{itemize}
\item[i)] $AA^*\in \cG\ell_\cJ$; 
\item[ii)] $(A^*)^\dagger \in \cL \cO_\cJ(A)$ and $(A^*)^\dagger \,A^*=I$;
\item [iii)] There exists $B\in\cL\cO_\cJ(A)$ such that $AB^*=I$, $BA^*|_{R(B)}=I_{R(B)}$;
\item[iv)] $|A^*|\in \cG\ell_\cJ$;
\item [v)] $V_A\in \cL \cO_\cJ(A)$ and $V_AV_A^*=I$;
\item [vi)] There is a co-isometry $V \in \cL \cO_\cJ(A)$; 
\item[vii)] There is a co-isometry $V \in \cO_\cJ(A)$;
\item [viii)] There is an isometry $V \in \cO_\cJ(A^*)$. 
\end{itemize}
\end{cor}
\begin{proof}
The result is a straightforward consequence of Theorem \ref{teo de ops para framesx28}.
\end{proof}

In the next result, we make use of strictly Schur-convex functions. We point out that its proof is based on several ideas from \cite{EC19}.

\begin{teo}\label{dist a isometrias y inversas parciales}
Let $\cJ=\fS_\Phi$ be a symmetrically-normed  ideal and let $A \in \cC \cR$ with polar decomposition $A=V_A|A|$. 
\ben
\item If $V\in\cO_\cJ(A)$  and $V\in\cP\cI$, then $\| A-V_A\|_\Phi\leq \|A-V\|_\Phi$. 
\item If $B\in\cL\cO_\cJ(A)$ is such that $AB^*|_{R(A)}=I_{R(A)}$, $BA^*|_{R(B)}=I_{R(B)}$, then $\|A-(A^*)^\dagger\|_\Phi\leq \|A-B\|_\Phi$.
\een
If we assume further that $\Phi$ is a strictly Schur-convex symmetric norming function and $\| A-V_A\|_\Phi= \|A-V\|_\Phi$, then $V=V_A$. Similarly, if $\|A-(A^*)^\dagger\|_\Phi=\|A-B\|_\Phi$, then $B=(A^*)^\dagger$.
\end{teo}

\begin{proof}
$1$. We first assume that $V\in\cL\cO_\cJ(A)$. By Theorem \ref{teo de ops para framesx28} and the hypothesis we get that $V_A\in\cL\cO_\cJ(A)$, and hence the norms $\|A-V_A\|_\Phi$ and $\|A-V\|_\Phi$ are well defined. Since $\cJ=\fS_\Phi$ is an arithmetic mean closed ideal then, by Corollary \ref{coro caract orb1} we get $V_A-V\in\cJ$ and $V_A^*V_A=V^*V$. Then, by item $i)$ in Remark \ref{unitary orb partial iso and proj}
we see that there exists $U\in\cU_\cJ$ such that $V=UV_A$. In this case
$A-V=V_A|A|-UV_A=U(U^*V_A|A|-V_A)$ and hence
$$
s(A-V)=s(U^*V_A|A|-V_A)=s(U^*V_A|A|V_A^*-V_AV_A^*)\,.
$$ If we let $C=V_A|A|V_A^*\in\cC\cR^+$, then $V_AV_A^*=P_{R(C)}$, $s(A-V)=s(U^*C-P_{R(C)})$ and $C-P_{R(C)}=(A-V_A)V_A^*\in\cJ$.
Since $U^*\in\cU_\cJ$ then $U^*C-P_{R(C)}=(U^*-I)C+C-P_{R(C)}\in\cJ\subset \cK$.
Using the fact that $C-P_{R(C)}\in\cK$ is a self-adjoint compact operator, we conclude that there exists an orthonormal basis $\{u_i\}_{i\in I}$ for $R(C)$ (here $I=\N$ or a finite set $I=\{1,\ldots,d\}$) such that 
$$C-P_{R(C)}= \sum_{i\in I} \la_i(C-P_{R(C)}) \, u_i\otimes u_i\,,$$
where the convergence is in the operator norm and $\{\la_i(C-P_{R(C)})\}_{i\in I}$ is an enumeration of the eigenvalues of $(C-P_{R(C)})|_{R(C)}\in \cK(R(C))$, counting multiplicities and such that 
$(|\la_i(C-P_{R(C)})|)_{i\in I}$ are arranged in non-increasing order. Also note that $u_i \otimes u_i$ stands for the rank-one orthogonal projection associated with the vector $u_i$. In this case,
$\lim_{i\rightarrow \infty} \la_i(C-P_{R(C)})=0$ if $I=\N$. If we let $\la_i(C)=\la_i(C-P_{R(C)})+1$ for $i\in I$, then $\lambda(C)=\{\la_i(C)\}_{i\in I}$ is an enumeration of the eigenvalues of the self-adjoint diagonalizable operator $C|_{R(C)}\in \cB(R(C))$ such that $C=\sum_{i\in I} \la_i(C) \, u_i\otimes u_i$, where the convergence is in the strong operator topology. Furthermore, we also get that $P_{R(C)}=\sum_{i\in I} u_i\otimes u_i$
where the convergence is also in the strong operator topology. We can now express
$$
U^*C-P_{R(C)}=\sum_{i\in I} \la_i(C) U^*u_i\otimes u_i - \sum_{i\in I} u_i\otimes u_i=
\sum_{i\in I} (\la_i(C) \,U^*u_i-u_i)\otimes u_i\,.
$$ For $k \geq 1$, we let $P_k=\sum_{i=1}^k  u_i\otimes u_i$. In this case,
\beq\label{eq hay conv en op norm1}
\sum_{i=1}^k (\la_i(C) U^*u_i-u_i)\otimes u_i=(U^*C-P_{R(C)})P_k\xrightarrow [k\rightarrow \infty]{}
U^*C-P_{R(C)}\,,
\eeq where the convergence is in the operator norm (\cite[Chap. III, Thm. 6.3]{GK60}). By construction $|(U^*C-P_{R(C)})P_k|$ is a finite rank positive operator and we can apply Lidskii's inequality for singular values of matrices (see, e.g.  \cite{Bhatia}) and conclude that 
\begin{align*}
\{s_i(C-P_{R(C)})\}_{i=1}^k &=\{ |\la_i(C-P_{R(C)})| \}_{i=1}^k=\{ |\la_i(C)-1| \}_{i=1}^k\\
&=\{ (|s_i(U^*CP_k)-s_i(P_k)|\}_{i=1}^k)\da \prec_w \{ s_i(\,(U^*C-P)\,P_k)\}_{i=1}^k\,,
\end{align*}
where we have used that $\{s_i(U^*CP_k)\}_{i=1}^k$ is a re-arrangement of $\{ \la_i(C)\}_{i=1}^k$, for $k \geq 1$.
By the norm convergence in Eq. \eqref{eq hay conv en op norm1}, the continuity of the singular values, and the previous submajorization relation we conclude that 
$s(C-P_{R(C)})\prec_w s(U^*C-P_{R(C)})$. By the Schur-convexity of $\Phi$ and the unitary invariance of $\|\cdot \|_\Phi$ we conclude that 
$$\|A-V_A\|_\Phi=\| V_A|A|V_A^*-V_AV_A^*\|_\Phi=\|C-P_{R(C)}\|_\Phi\leq \|U^*C-P_{R(C)} \|_\Phi=\|A-V\|_\Phi\,.$$
Assume further that $\Phi$ is a strictly Schur-convex norming function and $\| A-V_A\|_\Phi= \|A-V\|_\Phi$. 
Since $s(C-P_{R(C)})\prec_w s(U^*C-P_{R(C)})$ we conclude that $s(U^*C-P_{R(C)})=s(C-P_{R(C)})=\{|\la_i(C)-1|\}_{i\in I}$ (together with an infinite tail of zeros in case $I$ is a finite set). Hence,
$$
s^2(U^*C-P_{R(C)})=s(\,(U^*C-P_{R(C)}) (U^*C-P_{R(C)})^*)=s(\sum_{i\in I} (\la_i(C)\, U^*u_i-u_i)\otimes (\la_i(C)\, U^*u_i-u_i))\,.
$$ In this case it is well known (see \cite[Thm. 5.1]{AMRS}) that the sequence of norms of a sequence of vectors is submajorized by
the singular values (eigenvalues) of their frame operator, i.e.
$$
\{ \| \la_i(C)\, U^*u_i-u_i\|^2\} _{i\in I}\prec_w s\big(\sum_{i\in I} (\la_i(C)\, U^*u_i-u_i)\otimes (\la_i(C)\, U^*u_i-u_i)\big) =\{ |\la_i(C)-1|^2\} _{i\in I}\,.
$$In particular, for $k\in I$,
\beq \label{eq hay submayo21}
\sum_{i=1}^k \| \la_i(C)\, U^*u_i-u_i\|^2\leq \sum_{i=1}^k|\la_i(C)-1|^2 \,.
\eeq
On the other hand, for $i\in I$,
$$
|\la_i(C)-1|^2= \big | \|\la_i(C)\,U^*u_i\| - \|u_i\|\big|^2\leq \| \la_i(C)\, U^*u_i-u_i\|^2\,,
$$ with equality if and only if $U^*u_i=u_i$. The previous facts together with Eq. \eqref{eq hay submayo21} imply that $U^*u_i=u_i$ for $i\in I$, and hence $Uu=u$ for every $u\in R(C)=R(A)$; in particular, $V=UV_A=V_A$.

We now consider the general case $V\in\cO_\cJ(A)$. Since $\cJ=\fS_\Phi$ is an arithmetic mean closed ideal, then by Corollary \ref{coro caract orb2} we get $V_A-V\in\cJ$ and $[V_A^*V_A:V^*V]=0$. Thus, by item $ii)$ in Remark \ref{unitary orb partial iso and proj} we see that there exist $U,W\in\cU_\cJ$ such that $V=UV_AW$. Hence,
$A-V=A-UV_AW=(AW^*-UV_A)W$ so $s(A-V)=s(AW^*-UV_A)$. Notice that $AW^*=V_AW^* (W|A|W^*)$ is the polar decomposition of $AW^*$ so by the first part of the proof we conclude that 
$$\|A-V_A\|_\Phi=\|AW^*-V_AW^*\|_\Phi\leq \|AW^*-UV_A\|_\Phi =\|A-V\|_\Phi\,.$$
If we assume further that $\Phi$ is a strictly Schur-convex norming function and $\|A-V_A\|_\Phi=\|A-V\|_\Phi$, then the previous inequalities show that $\|AW^*-V_AW^*\|_\Phi=\|AW^*-UV_A\|_\Phi$. By the first part of the proof we get that
$V_AW^*=UV_A$ so then $V_A=UV_AW=V$.

\medskip

\noindent $2.$ Let $B\in\cL\cO_\cJ(A)$ be such that $AB^*|_{R(A)}=I_{R(A)}$, $BA^*|_{R(B)}=I_{R(B)}$. In this case Theorem \ref{teo de ops para framesx28} implies that $(A^*)^\dagger\in\cL\cO_\cJ(A)$, so the norms $\|A-B\|_\Phi$ and $\|A-(A^*)^\dagger \|_\Phi$ are well defined. The identity $AB^*|_{R(A)}=I_{R(A)}$ implies 
$$P_{N(A)^\perp} B^* P_{R(A)} = A^\dagger A B^* P_{R(A)}=A^\dagger\,.$$
In particular, 
$$s(A-(A^*)^\dagger)=s(A-P_{R(A)} B P_{N(A)^\perp})=s(P_{R(A)}(A- B) P_{N(A)^\perp})\prec_w s(A-B)\,,$$
where the last submajorization relation follows from the fact that 
$s_i(P_{R(A)}(A- B) P_{N(A)^\perp})\leq s_i(A-B)$, for $i \geq 1$. Hence, $\|A-(A^*)^\dagger\|_\Phi\leq \|A-B\|_\Phi$. 

Assume further that $\Phi$ is a strictly Schur-convex symmetric norming function and $\|A-(A^*)^\dagger\|_\Phi= \|A-B\|_\Phi$. Hence, in this case we get that $s(P_{R(A)}(A- B) P_{N(A)^\perp})= s(A-(A^*)^\dagger)=s(A-B)$. Since
$A-B\in\cK(\cH)$, then the last identity between singular values implies that 
$R(A-B)\subset R(A)$ and $R(A^*-B^*)=N(A-B)^\perp\subset N(A)^\perp=R(A^*)$ (see the comments in Remark \ref{ya casi} below). This last inclusions show that $R(B)\subset R(A)$ and $N(B)^\perp=R(B^*)\subset R(A^*)=N(A)^\perp$ 
Hence $B=P_{R(A)}BP_{N(A)^\perp}=(A^*)^\dagger$.
\end{proof}

\begin{rem}\label{ya casi}
Let $T\in\cK$ and let $P,Q$ be orthogonal projections such that 
$s(PTQ)=s(T)$. Then, we have $s_i(TT^*)=s_i(PTQT^*P)\leq s_i (PTT^*P) \leq s_i(TT^*)$ for $i \geq 1$, so that $s(TT^*)=s(PTT^*P)$. According to \cite[Chap. II, Thm. 5.1]{GK60}, we get that $PTT^*P=TT^*$. Therefore,  $PTT^*P(T^*)^\dagger=T$, which gives $R(T) \subset R(P)$. Similarly, one can show that $N(T)^\perp=\overline{R(T)} \subset R(Q)$.
\end{rem}

\section{Orbits of the restricted unitary group}\label{sec4}

In this section, we investigate the corresponding two types of orbits defined by the action of restricted unitary groups. 
More precisely, let $\cJ$ be an operator ideal, and take $A \in \cC\cR$. We then consider the corresponding restricted unitary orbits given by 
$$
\cL\cU_\cJ(A)=\{ UA : U \in \cU_\cJ \}
\py
\cU_\cJ(A)=\{ UAW^* : U, W \in \cU_\cJ \}.
$$

\subsection{Spatial characterizations}

We begin with the analog spatial characterizations for these unitary orbits.

\begin{pro}\label{teo rest orb unit}
Let $A,B \in\cC\cR$ with polar decompositions $A=V_A|A|$, $B=V_B|B|$, and let $\cJ$ be a proper operator ideal.
Then, the following are equivalent:
\ben 
\item[i)] There exists $U\in\cU_\cJ$ such that $UA=B$;
\item[ii)] $A-B\in\cJ$ and $|A|=|B|$;
\item[iii)] $V_A-V_B\in\cJ$ and $|A|=|B|$. In this case, $V_B\in\cU_\cJ(V_A)$.
\een
If any of the conditions above holds then $V_B\in\cU_\cJ(V_A)$ so $[P_{R(A)}:P_{R(B)}]=0$.
\end{pro}
\begin{proof}
$i) \rightarrow ii)$. If there exists $U\in\cU_\cJ$ such that $UA=B$ then we have that  $A-B\in\cJ$ and 
$B^*B=(UA)^*(UA)=A^*A$. This implies that $|A|=|B|$.

\medskip

\noi $ii) \rightarrow iii)$.  Assume that $|A|=|B|$ and $A-B\in\cJ$. Then, 
$$
V_A-V_B= A\,|A|^\dagger -B\,|B|^\dagger=A\,|A|^\dagger-B\,|A|^\dagger=(A-B)\,|A|^\dagger\in\cJ\,.
$$

\noi $iii) \rightarrow i)$.  
Notice that $N(V_A)=N(A)=R(|A|)^\perp=R(|B|)^\perp=N(B)=N(V_B)$. 
From item $i)$ in Remark \ref{unitary orb partial iso and proj} we see that 
there exists $U\in\cU_\cJ$ such that $UV_A=V_B$. Hence, we conclude that 
$UA=(UV_A)|A|=V_B|B|=B$.
\end{proof}

We can now state our first main result about the orbit of operators given by multiplication on both sides by the group $\cU_\cJ$.

\begin{teo}\label{teo biorb unit res}
Let $\cJ$ be a proper operator ideal.
Let $A,\,B\in \cC\cR$ with polar decompositions $A=V_A\,|A|$ and $
B=V_B\,|B|$.  
Then, the following are equivalent:
\ben 
\item[i)] There exist $U,\,W\in\cU_\cJ$ such that $UAW^*=B$;
\item[ii)] There exist $U,\,W,\,Z\in\cU_\cJ$ such that $W|A|W^*=|B|$ and $UV_AZ^*=V_B$;
\item[iii)] There exists $W\in\cU_\cJ$ such that $W|A|W^*=|B|$ and $V_A-V_B\in\cJ$, $[\, P_{R(A)}:P_{R(B)}\,]=0$.
\een
\end{teo}
\begin{proof}
$i) \rightarrow ii).$ If there exist $U,\,W\in\cU_\cJ$ such that $UAW^*=B$, then
$$
V_B\,|B|=UAW^*=(UV_AW^*)(W|A|W^*)=VC
$$
where $V=UV_AW^*$ is a partial isometry and $C=W|A|W^*$ is a positive operator. 
Since $R(V^*)=R(WV_A^*)=R(W|A|W^*)=R(C)$ then, by the uniqueness in the polar decomposition, we get that $UV_AW^*=V=V_B$ and $W|A|W^*=C=|B|$. That is, we see that item $ii)$ holds with $Z=W$.

\medskip

\noi $ii)\rightarrow i).$ Assume that 
there exist $U,\,W,\, Z\in\cU_\cJ$ such that $W|A|W^*=|B|$ and $UV_AZ^*=V_B$.
In particular, 
$$
B=V_B\,|B|=(UV_AZ^*)(W|A|W^*)=U(V_AZ^*WV_A^*)V_A|A|W^*=U(V_AZ^*WV_A^*)AW^*\,.
$$
On the one hand, the identity $W|A|W^*=|B|$ implies that 
$$R(WV_A^*)=WR(V_A^*)=WR(|A|)=R(W|A|)=R(|B|)=R(V_B^*)\,.$$
On the other hand, the identity $UV_AZ^*=V_B$ implies that 
$V_AZ^*|_{R(V_B^*)}=U^*V_B|_{R(V_B^*)}:R(V_B^*)\rightarrow R(U^*V_B)=R(V_AZ^*)=R(V_A)$ 
is a well defined isometric isomorphism between $R(V_B^*)$ and 
$R(V_A)$. These last facts show that 
the restriction $V_A Z^*WV_A^*|_{R(V_A)}:R(V_A)\rightarrow R(V_A)$ is an 
isometric isomorphism. Hence, if we let $Y=V_AZ^*WV_A^*+P_{N(A^*)}\in \cB(\cH)$, then by construction $Y\in \cU(\cH)$ and
$$
Y-I=V_AZ^*WV_A^*+P_{N(A^*)} - (V_AV_A^*+P_{N(A^*)})=V_A( Z^*W-I)V_A^*\in\cJ\,.
$$ 
Finally, notice that $YA=(V_AZ^*WV_A^*)\,A$ and hence,
$B=UYAW^*$, with $UY,\,W\in\cU_\cJ$.

\medskip

\noi $ii) \leftrightarrow iii).$  This is a consequence of item $ii)$ in Remark \ref{unitary orb partial iso and proj}. 
\end{proof}

\subsection{Restricted diagonalization and block singular value decomposition}

As we have observed in Remark \ref{unitary orb partial iso and proj} there has been interest in restricted unitary orbits of partial isometries. These operators can be identified with operators with two singular values. Hence, it is natural to consider the context of restricted unitary orbits of operators with a finite number of singular values. More generally, we can consider the class of operators whose modulus are diagonalizable operators.

\begin{rem}[Block singular value decomposition for operators with diagonalizable modulus]\label{rem sing val dec}
 Recall that an operator $T \in \cB(\cH)$ is diagonalizable if there exists an orthonormal basis $\{ e_n\}_{n \geq 1}$ of $\cH$ such that $\PI{Te_n}{e_n}=\delta_{nm} \lambda _{mn}$ for some bounded sequence of complex numbers $\{ \lambda_n\}_{n \geq 1}$. For instance, $T$ is diagonalizable whenever $T$ has finite spectrum. 
 
 Let $A \in \cC \cR$ be such that $|A|$ is diagonalizable. In this case, if we let $\sigma_p(|A|)$ denote the point spectrum of $|A|$ and $P_\la$ denote the spectral projection of $|A|$ associated to $\la\in\sigma_p(|A|)$,  
$$
|A|=\sum_{\la\in\sigma_p(|A|)}\la\, P_\la=\sum_{\la\in\sigma_p(|A|)\setminus\{0\}}\la\, P_\la \peso{,} I= \sum_{\la\in\sigma_p(|A|)}P_\la
\py P_{R(|A|)}=\sum_{\la\in\sigma_p(|A|)\setminus\{0\}}P_\la \, ,
$$ where 
the series converges in the SOT. Let $A=V_A\,|A|$ be the polar decomposition of $A$ (so that $R(V_A^*)={R(|A|)}$). If we let $V_\la=V_A\,P_\la$ for $\la\in\sigma_p(|A|)\setminus\{0\}$, then $\{V_\la\}_{\la\in\sigma_p(A)\setminus\{0\}}$ is a family of partial isometries satisfying the following conditions:
\ben
\item $A=\sum_{\la\in\sigma_p(|A|)\setminus\{0\}}\la\, V_\la$, where the series converges in the SOT;
\item $\{V_\la \,V_\la^*\}_{\la\in\sigma_p(|A|)\setminus\{0\}}$ (respectively $\{V_\la^*V_\la\}_{\la\in\sigma_p(|A|)\setminus\{0\}}$) is a family of non-zero, mutually orthogonal projections; 
\item $\sum_{\la\in\sigma_p(|A|)\setminus\{0\}}V_\la\,V_\la^*=P_{R(A)}$ and $\sum_{\la\in\sigma_p(|A|)\setminus\{0\}}V_\la^*V_\la=P_{R(A^*)}$.
\een
Moreover, the previous properties completely determine the family $\{V_\la\}_{\la\in\sigma_p(|A|)\setminus\{0\}}$.
Notice that this last representation of $A$ can be considered as a block singular value decomposition.
\end{rem}

The following result is based on ideas from \cite{ECPM1} and it will be needed in the proof of Theorems \ref{teo caso diagonal0} and \ref{teo caso diagonal2} below.

\begin{pro}\label{pro caract de iso parc con igual esp inic}
Let $\{V_j\}_{j=1}^N$, $\{W_j\}_{j=1}^N$ (with $N\in\N$ or $N=\infty$) be families of partial isometries in $\cB(\cH)$ such that 
each 
$$
\{V_j V_j^*\}_{j=1}^N \ \ , \ \ \{V_j^*V_j=W_j^*W_j\}_{j=1}^N  \py \{W_jW_j^*\}_{j=1}^N
$$  
is a family of mutually orthogonal projections. Then the following conditions are equivalent:
\ben
\item[i)] There exists $Z\in\cU_\cJ$ such that $ZV_j =W_j$, for $j \in \N$; 
\item[ii)] The series 
\beq \label{las cond}
\sum_{j= 1}^N (W_j-V_j)V_j^*\in\cJ \py \sum_{j=1}^N (V_j-W_j)W_j^*\in\cJ\,,
\eeq where the convergence is in the operator norm when $N=\infty$.
\een
\end{pro}
\begin{proof} 
$i)\rightarrow ii)$. Assume that there exists $Z\in\cU_\cJ$ such that $ZV_j=W_j$, for $j\geq 1$. Put $P=\sum_{j= 1}^N V_jV_j^*$. Then, 
$$
\sum_{j=1}^N (W_j-V_j)V_j^*=\sum_{j=1}^N (ZV_j-V_j)V_j^*=\sum_{j=1}^N (Z-I)V_jV_j^*=(Z-I)P\in\cJ,
$$ 
 where the series converge in operator norm, since $Z-I \in\cJ\subseteq \cK$ and $\{V_jV_j^*\}_{j= 1}^N$ is a family of mutually orthogonal projections. Similarly, one shows that $\sum_{j=1}^N (V_j-W_j)W_j^*\in\cJ$.  Hence, the conditions in Eq. \eqref{las cond} are satisfied.

\medskip

\noi $ii)\rightarrow i)$. We denote $V_jV_j^*=P_j$, $W_jW_j^*=Q_j$, $j\geq 1$. We also take $P=\sum_{j= 1}^N P_j$ and $Q=\sum_{j= 1}^N Q_j$. Next, we consider the case $N=\infty$; the case $N<\infty$ follows similarly.   If we let $f\in\cH$ and $n\in\N$, then
$$
\|\sum_{j\geq n}W_jV_j^*f \|^2 =\sum_{j\geq n}\left\|W_jV_j^*f \right\|^2 \leq \sum_{j\geq n}\left\|V_j^*f \right\|^2\xrightarrow[n\rightarrow \infty]{}0\,.
$$
Hence, we set
$$
S:=\sum_{j= 1}^\infty W_jV_j^* \, ,
$$ 
where the series converges in the SOT. By hypothesis, we get that 
$SV_j=W_j$ y $S^*W_j=V_j$, $j\geq 1$.
On the other hand, using that 
$S$ and $S^*$ are SOT-limits of the uniformly bounded sequences 
$S_n=\sum_{j=1}^n W_jV_j^*$ and $S_n^*$ respectively,
 and the fact that the product is SOT continuous on bounded sets, we get that
$$
S^*S=\sum_{j= 1}^\infty V_jW_j^*W_jV_j^*=\sum_{j= 1}^\infty V_jV_j^*=P
$$ where we used that $V_jW_j^*W_j=V_j$, $j\geq 1$. 
Similarly, we can show that $SS^*=Q$. Hence, $S$ is a partial isometry with initial space $R(P)$ and
final space $R(Q)$. 
By Eq. \eqref{las cond} we have that 
$$
S-P=\sum_{j= 1}^\infty W_jV_j^*-\sum_{j= 1}^\infty V_jV_j^*=\sum_{j= 1}^\infty (W_j-V_j)V_j^*\in\cJ\,,
$$
Since the partial isometries $S$ and $P$ have the same null space and are such that $S-P\in\cJ$ then,
by item $i)$ in Remark \ref{unitary orb partial iso and proj}, we get that there exists $Z\in\cU_\cJ$ such that $ZP=S$. Therefore, $ZV_j=ZPV_j=SV_j=W_j$, for all $j\geq 1$.
\end{proof}

\begin{teo}\label{teo caso diagonal0}
Let $A,B \in\cC\cR$ and $\cJ$ be a proper operator ideal.  Assume that $|A|$ is a diagonalizable operator
and let $\{V_\la\}_{\la \in \sigma_p(|A|)\setminus\{0\}}$ denote the partial isometries of the block singular value decomposition of $A$.
Then, the following are equivalent:
\ben 
\item[i)] There exists $U\in\cU_\cJ$ such that $UA=B$;
\item[ii)] $|B|$ is diagonalizable, $\sigma_p(|B|)=\sigma_p(|A|)$, and if we let $\{W_\la\}_{\la \in \sigma_p(|A|)\setminus\{0\}}$ denote the partial isometries of the block singular value decomposition of $B$, then $V_\la^* V_\la=W_\la^* W_\la$, for $\la \in \sigma_p(|A|)\setminus\{0\}$,  
\beq \label{las cond2}
\sum_{\la \in \sigma_p(|A|)\setminus\{0\}} (W_\la-V_\la)V_\la^*\in\cJ \py \sum_{\la \in \sigma_p(|A|)\setminus\{0\}} (V_\la-W_\la)W_\la^*\in\cJ\,,
\eeq where the convergence is in the operator norm.
\een
\end{teo}
\begin{proof}
$i)\rightarrow ii).$ Assume that there exists $U\in\cU_\cJ$ such that $UA=B$. By Proposition \ref{teo rest orb unit} and the hypothesis above we see that $|A|=|B|$ is a diagonalizable operator. Moreover, arguing as in the proof of Proposition \ref{teo rest orb unit} we also see that $UV_A=V_B$, where $A=V_A|A|$ and $B=V_B|B|$ are the polar decompositions of $A$ and $B$, respectively. Hence, if we let $\{P_\la\}_{\la\in\sigma_p(|A|)}$ denote the spectral projections of $|A|$, then by Remark \ref{rem sing val dec} we get that  
$$W_\la=V_BP_\la=UV_A P_\la=UV_\la\peso{for} \la\in \sigma_p(|A|)\setminus\{0\}\,.$$ Hence, 
$V_\la^* V_\la=W_\la^* W_\la$, for $\la \in \sigma_p(|A|)\setminus\{0\}$ and
the conditions in Eq. \eqref{las cond2} are a consequence of Proposition \ref{pro caract de iso parc con igual esp inic}.

\medskip

\noi $ii)\rightarrow i).$ In this case we can apply Proposition \ref{pro caract de iso parc con igual esp inic} and conclude that there exists $U\in\cU_\cJ$ such that $UV_\la=W_\la$, for $\la\in\sigma_p(|A|)\setminus\{0\}=\sigma_p(|B|)\setminus\{0\}$. Using the block singular value representation for $A$ and $B$ we see that 
$$
UA=\sum_{\la\in\sigma_p(|A|)\setminus\{0\}} \la \ UV_\la=\sum_{\la\in\sigma_p(|A|)\setminus\{0\}} \la W_\la=B\,.
 \qedhere$$
\end{proof}

As opposed to the $\cJ$-congruence class of a positive closed range operator, the $\cJ$-restricted unitary orbit of a positive closed range operator contains operators that have mutually strong structural relations. To take advantage of these structural relations we consider the following result, which is a consequence of \cite{ECPM1} and it will be needed in the proof of Theorem \ref{teo caso diagonal2}.

\begin{pro}\label{pro diag rest normales1}Let $\cJ$ be a proper arithmetic mean closed operator ideal.
Let $A,\,B\in \cB(\cH)$ be normal operators, and assume that $A$ is diagonalizable. Then, the following are equivalent:
\ben 
\item[i)] There exists $U\in\cU_\cJ$ such that $U^*AU=B$;
\item[ii)] $B$ is diagonalizable, $\sigma_p(A)=\sigma_p(B)$, and if we let $P_\la$ and $Q_\la$ denote the spectral projections of $A$ and $B$ associated with $\la\in\sigma_p(A)$, then
\beq \label{eq cond proj espec1}
\sum_{\la\in\sigma_p(A)} P_\la (I-Q_\la)\in \cJ\py\sum_{\la\in\sigma_p(A)} (I-P_\la) Q_\la\in \cJ\,,
\eeq where the series converge weakly. Furthermore, $[P_\la:Q_\la]=0$, for $\la\in\sigma_p(A)$.
\een
\end{pro}

\begin{proof}
$i) \rightarrow ii)$. In this case, $B$ is also diagonalizable, since $A$ is diagonalizable by assumption. Also, in this case, we get that $\sigma_p(A)=\sigma_p(B)$. Moreover, 
we have that $U^*P_\la U=Q_\la$ for every $\la\in\sigma_p(A)$. Hence, by  \cite[Corollary 3.15]{ECPM1}, we conclude that the condition in Eq. \eqref{eq cond proj espec1} holds; moreover, we also get that $[P_\la:Q_\la]=0$, for $\la\in\sigma_p(A)$. 

\medskip

\noi $ii)\rightarrow  i)$. Using Corollary \cite[Corollary 3.15]{ECPM1} we conclude that there exists $U\in\cU_\cJ$ such that $U^*P_\la U=Q_\la$ for every $\la\in\sigma_p(A)$. From this last fact, we see that 
$$U^*AU= U^*\left(\sum_{\la\in\sigma_p(A)} \la\,P_\la\right)U= \sum_{\la\in\sigma_p(A)}\la \, U^* P_\la U=
\sum_{\la\in\sigma_p(B)} \la\,Q_\la=B\,. \qedhere
$$
\end{proof}

\begin{teo}\label{teo caso diagonal2}
Let $\cJ$ be a proper arithmetic mean closed operator ideal.
Let $A,\,B\in \cC\cR$ be  operators with polar decompositions $A=V_A\,|A|$ and $
B=V_B\,|B|$. Assume that $|A|$ is diagonalizable.
Then, the following are equivalent:
\ben 
\item[i)] There exist $U,\,W\in\cU_\cJ$ such that $UAW^*=B$;
\item[ii)] The following two conditions hold:
\ben
\item[a)] $|B|$ is diagonalizable, $\sigma_p(|B|)=\sigma_p(|A|)$, and if we let $P_\la$ and $Q_\la$ denote the spectral projections of $|A|$ and $|B|$ associated with $\la\in\sigma_p(|A|)$, then 
\beq \label{eq cond proj espec2}
\sum_{\la\in\sigma_p(|A|)} P_\la (I-Q_\la)\in \cJ\py\sum_{\la\in\sigma_p(|A|)} (I-P_\la) Q_\la\in \cJ\,,
\eeq where the series converge weakly and $[P_\la:Q_\la]=0$ for $\la\in\sigma_p(|A|)$.
\item[b)] $V_A-V_B\in\cJ$.
\een
\item[iii)] The condition in item $ii)\, a)$ holds, and if we let $\{V_\la\}_{\la\in\sigma_p(|A|)\setminus\{0\}}$ and $\{W_\la\}_{\la\in\sigma_p(|A|)\setminus\{0\}}$ denote the partial isometries of the block singular decomposition of $A$ and $B$, then 
\beq \label{las cond3}
\sum_{\la\in\sigma_p(|A|)\setminus\{0\}} (W_\la-V_\la)V_\la^*\in\cJ \py \sum_{\la\in\sigma_p(|A|)\setminus\{0\}}  (V_\la-W_\la)W_\la^*\in\cJ\,,
\eeq where the series converge in the operator norm.
\een
\end{teo}
\begin{proof}
$i)\rightarrow ii)$. Assume first that there exist $U,\,W\in\cU_\cJ$ such that $UAW^*=B$. Then from Theorem \ref{teo biorb unit res} we get that $W|A|W^*=|B|$ and $UV_AZ^*=V_B$. In particular, $|B|$ is also diagonalizable and 
$\sigma_p(|B|)=\sigma_p(|A|)$. Taking into account that $\cJ$ is an arithmetic mean closed proper ideal, we see that the conditions in item $ii)\, a)$ follow from Proposition \ref{pro diag rest normales1}.
On the other hand, the conditions in item $ii) \, b)$  follow from Remark \ref{unitary orb partial iso and proj} $ii)$.

\medskip

\noi $ii) \rightarrow i).$ In this case we can apply Proposition \ref{pro diag rest normales1} and 
conclude that there exists $W\in\cU_\cJ$ such that $W|A|W^*=|B|$. In particular, $P_{R(B^*)}= P_{R(|B|)}=
W P_{R(|A|)}W^*= WP_{R(A^*)}W^*$. Hence, $[V_A^*V_A: V_B^*V_B]= [P_{R(A^*)}:P_{R(B^*)}]=0$. Since $V_A-V_B\in\cJ$ then item $ii)$ in
Remark \ref{unitary orb partial iso and proj} shows that there exist $U,\,Z\in\cU_\cJ$ such that $UV_AZ^*=V_B$. Then, item $i)$ follows from Theorem \ref{teo biorb unit res}.

\medskip

\noi $i)\rightarrow iii)$. As in the first part of the proof, the conditions in item $i)$ imply the conditions in item $ii) \, a)$. On the other hand, by hypothesis, we get that
$$
B=UAW^*=U\left(\sum_{\la\in\sigma_p(|A|)\setminus\{0\}} \la\ V_\la \right)W^*=\sum_{\la\in\sigma_p(|A|)\setminus\{0\}} \la\ UV_\la W^*\,.
$$ 
By uniqueness of the block singular value decomposition  (see Remark \ref{rem sing val dec}), we see that if we let $\tilde V_\la=V_\la W^*$, then $W_\la=U\tilde V_\la$, for $\la\in\sigma_p(|A|)\setminus\{0\}$.
By Proposition \ref{pro caract de iso parc con igual esp inic} we conclude that 
$$
\sum_{\la\in\sigma_p(|A|)\setminus\{0\}} (W_\la-\tilde V_\la)\tilde V_\la^*\in\cJ \py \sum_{\la\in\sigma_p(|A|)\setminus\{0\}}  (\tilde V_\la-W_\la)W_\la^*\in\cJ\,,
$$ where the series converge in operator norm. 

We claim that the conditions in Eq. \eqref{las cond3} hold in this case.
Indeed, since $W\in\cU_\cJ$, there exists $K\in\cJ$ such that $W=I+K$. In this case,
$$
\sum_{\la\in\sigma_p(|A|)\setminus\{0\}} (W_\la-\tilde V_\la)\tilde V_\la^*=
\sum_{\la\in\sigma_p(|A|)\setminus\{0\}} W_\la \,K^*\, V_\la^*+
\sum_{\la\in\sigma_p(|A|)\setminus\{0\}} (W_\la- V_\la)V_\la^*
$$
where the first series to the right converge in the operator norm. Indeed, first notice that 
$$
0\leq \left(\sum_{\la\in\sigma_p(|A|)\setminus\{0\}} W_\la \,K^*\, V_\la^*\right)^* \left(\sum_{\la\in\sigma_p(|A|)\setminus\{0\}} W_\la \,K^*\, V_\la^* \right) 
\leq 
\sum_{\la\in\sigma_p(|A|)\setminus\{0\}} V_\la \,KK^*\, V_\la^*\in\cJ^2
$$ since $K^*K\in\cJ^2$ and $\cJ^2$ is (also) arithmetic mean closed, so that the pinching operator of elements in $\cJ^2$ lie in $\cJ^2$ (see, e.g. \cite{GK60}). Moreover, the convergence of the series is in the operator norm since
$KK^*$ is a compact operator and each of the sequences $\{V_\la V_\la^*\}_{\la\in\sigma_p(A)}$ and $\{V_\la^*V_\la\}_{\la\in\sigma_p(A)}$ have mutually orthogonal elements. As a consequence, we see that $\sum_{\la\in\sigma_p(|S|)} W_\la \,K\, V_\la^*\in\cJ$, where the series converges in the operator norm. This proves that the first condition in Eq. \eqref{las cond3} holds; the second condition follows by a similar argument.

\medskip

\noi $iii)\rightarrow i)$.  Assume that the conditions in item $ii) \, a)$ and Eq. \eqref{las cond3} hold. Arguing as in the proof of Proposition \ref{pro diag rest normales1} we see that there exists $W\in\cU_\cJ$ such that $WP_\la W^*=Q_\la$ so $R(W^*Q_\la)=R(P_\la)$, for $\la\in\sigma_p(|A|)$. Hence, if we let $\tilde V_\la=V_\la W^*$ then $\tilde V_\la$ is a partial isometry with 
$\tilde V_\la^* \tilde V_\la=WP_\la W^*=Q_\la=W_\la^*W_\la$, for $\la\in\sigma_p(|A|)\setminus\{0\}$.
Arguing as in the proof of $i)\rightarrow iii)$ above, the previous facts imply that 
$$
\sum_{\la\in\sigma_p(|A|)\setminus\{0\}} (W_\la-\tilde V_\la)\tilde V_\la^*\in\cJ \py \sum_{\la\in\sigma_p(|A|)\setminus\{0\}}  (\tilde V_\la-W_\la)W_\la^*\in\cJ\,,
$$ where the series converge in the operator norm. Therefore, we can apply Proposition \ref{pro caract de iso parc con igual esp inic} and conclude that there exists $U\in\cU_\cJ$ such that $U\tilde V_\la=W_\la$, for $\la\in \sigma_p(|A|)\setminus\{0\}$. Then, we get that 
$$
UAW^*=U\left(\sum_{\la\in\sigma_p(|A|)\setminus\{0\}} \la\ V_\la \right)W^*=\sum_{\la\in\sigma_p(|A|)\setminus\{0\}} \la\ U\,\tilde V_\la=
\sum_{\la\in\sigma_p(|A|)\setminus\{0\}} \la\ W_\la = B\,. \qedhere
$$
\end{proof}

\section{Restricted equivalences of frames}\label{sec res equiv frames}

We now introduce the notion of $\cJ$-equivalent and $\cJ$-unitarily equivalent frames for subspaces, and derive several characterizations. The notion of $\cJ$-unitarily equivalence between frames naturally extends the $\cJ$-equivalent orthonormal bases described in \cite{BPW16}. We first consider the following elementary facts about frame theory (for a detailed account see \cite{Ca00, Chris16}). As before, $\cH$ denotes a separable infinite-dimensional complex Hilbert space.

\subsection{Frames for subspaces: elementary theory}

\begin{fed}\label{defi basicas frames}\label{sec frames elem teo}
Let $\cS\subset \cH$ be a closed subspace of $\cH$ and let $\cF=\{f_n\}_{n \geq 1}$ be a sequence in $\cS$. 
\begin{enumerate}
\item We say that $\cF$ is a \textit{frame for $\cS$} if there exist constants $0<\alpha\leq \beta $ such that 
$$\alpha\|f\|^2\leq \sum_{n \geq 1} |\langle f\coma f_n\rangle|^2 \leq \beta\|f\|^2 \peso{for every} f\in\cS\,.$$
If the upper bound to the right holds for some $\beta \geq 0$ then we say that $\cF$ is a \textit{Bessel sequence}.
\item If $\cF$ is a Bessel sequence, then the \textit{synthesis operator} of $\cF$, denoted by $T_\cF\in \cB(\ell^2(\N),\cH)$, is uniquely determined by 
$$
T_\cF(e_n)=f_n \,  , \ \ n \geq 1,
$$ 
where $\{e_n\}_{n \geq 1}$ is the canonical orthonormal basis of $\ell^2(\N)$. In this case $T_\cF^*\in \cB(\cH,\ell^2 (\N))$ is the \textit{analysis operator} of $\cF$.
\item If $\cF$ is a Bessel sequence then we define its \textit{frame operator}, denoted by $S_\cF\in \cB(\cH)^+$, and given by $S_\cF:=T_\cF T_\cF^*$.
\end{enumerate}
\end{fed}

In the sequel, we will identify $\ell^2(\N)$ with $\cH$ through some fixed orthonormal basis of $\cH$, that we also denote $\{e_n\}_{n \geq 1}$, by abuse of notation. Thus, we consider $T_\cF\in \cB(\cH)$.

\begin{rem}
Let $\cS\subset \cH$ be a closed subspace of $\cH$ and let $\cF=\{f_n\}_{n \geq 1}$ be a sequence in $\cS$. Then $\cF$ is a frame for $\cS$ if and only if $\cF$ is a Bessel sequence such that $R(T_\cF)=\cS$. Similarly, $\cF$ is a frame for $\cS$ if and only if $\cF$ is a Bessel sequence such that $R(S_\cF)=\cS$. In this case, $T_\cF$ and $S_\cF$ are closed range operators and the restriction $S_\cF|_\cS$ is a positive invertible operator acting on $\cS$.
\end{rem} 

One of the fundamental properties of a frame $\cF$ for a subspace $\cS$ is that it allows to represent vectors $f\in \cS$  as linear combinations of elements in $\cF$ in a {\it stable} way.  To formalize these facts we recall the notion of oblique duality (for more details see \cite{Chris16,ChrEld,Eldar14}).

\begin{fed}\label{rem sobre duales}
Let $\cS,\,\cT\subset \cH$ be closed subspaces such that $\cS+\cT^\perp=\cH$ and $\cS\cap \cT^\perp=\{0\}$. Let $\cF=\{f_n\}_{n \geq 1}$ and $\cG=\{g_n\}_{n \geq 1}$ be frames for $\cS$ and $\cT$,  respectively. We say that $\cG$ is an \textit{oblique dual of $\cF$} if $$f=\sum_{n\ \geq 1} \langle f,g_n\rangle f_n \py g=\sum_{n \geq 1} \langle g,f_n\rangle g_n \peso{for every} f\in\cS \ , \ g\in\cT\,.$$
Equivalently, $\cG$ is an oblique dual of $\cF$ if  $$ T_\cF T_\cG^*|_{\cS}=I_{\cS} \py T_\cG T_\cF^*|_{\cT}=I_{\cT}\,. $$
\end{fed}

\begin{rem} Let $\cS$ be a closed subspace of $\cH$ and 
let $\cF=\{f_n\}_{n \geq 1}$ be a frame for $\cS$. Then $\cF$ always admits oblique dual frames. Indeed, take $\cT=\cS$ and define $\cF^\# =\{ f_n^\# \}_{n \geq 1}$ given by $f^\#_n=S_\cF^\dagger f_n$, for $n \geq 1$. Notice that $T_{\cF^\#}=S_\cF^\dagger T_\cF=(T_\cF T_\cF^*)^\dagger T_\cF= (T_\cF^*)^\dagger=(T_\cF^\dagger )^*$ and hence $R(T_{\cF^\#})=\cS$. Thus, $\cF^\#$ is a frame for $\cS$ such that, for $f\in\cS$ we have  
$$ 
\sum_{n \geq 1}\langle f, f_n^\#\rangle f_n=T_\cF (T_{\cF^\#}^* f)=T_\cF T_\cF^\dagger f= P_\cS f=f 
$$
and
$$
\sum_{n \geq 1}\langle f, f_n\rangle f_n^\#=T_{\cF^\#} (T_{\cF}^* f)=(T_\cF^*)^\dagger T_\cF^* f= P_\cS f=f\,. 
$$
Notice that we have used the definitions of the synthesis and analysis operators of $\cF$ and $\cF^\#$ above. We call $\cF^\#$ {\it the canonical dual of $\cF$};  $\cF^\#$ has several structural properties related to $\cF$.
\end{rem}

\begin{rem}
Consider the notation in Remark \ref{rem sobre duales}. Notice that to construct an oblique dual for $\cF$, we have to compute an inverse of the action of the synthesis/analysis operator of $\cF$.  For example, to compute the canonical dual $\cF$ we have to compute the Moore-Penrose  inverse of $T_{\cF}^*$. 
\end{rem}
The comments in the previous remark motivate the introduction of a class of frames that are self-duals as follows.

\begin{fed}
Let $\cS\subset \cH$ be a closed subspace of $\cH$ and let $\cF=\{f_n\}_{n \geq 1}$ be a frame for $\cS$. We say that $\cF$ is a \textit{Parseval frame for $\cS$} if $\cF$ is a (oblique) dual frame for $\cF$,  i.e.
$$
f=\sum_{n \geq 1} \langle f,f_n\rangle \ f_n\peso{for every} f\in\cS\,.
$$ Equivalently, $\cF$ is a Parseval frame for $\cS$ if for every $f\in\cS$ we have that $\|f\|^2=\sum_{n\geq 1}|\langle f,f_n\rangle |^2$.
\end{fed}

\begin{rem}\label{rem defi pars asoc}
Let $\cS\subset \cH$ be a closed subspace of $\cH$ and let $\cF=\{f_n\}_{n \geq 1}$ be a frame for $\cS$. Then $\cF$ is a Parseval frame for $\cS$ if $T_\cF T_\cF^*|_{\cS}=I_{\cS}=I_{R(T_{\cF})}$. That is, if $T_\cF$ is a partial isometry. 
Given an arbitrary frame $\cG=\{g_n\}_{n \geq 1}$ for $\cS$, then we define its  {\it associated Parseval frame}, denoted by $\tilde \cG=\{\tilde g_n\}_{n \geq 1}$, given by $\tilde g_n= |T_\cG^*| ^\dagger g_n$, for $n \geq 1$. Notice that $T_{\tilde \cG}=|T_\cG^*| ^\dagger T_\cG=V_\cG$, where $T_\cG=V_\cG |T_\cG|$ is the polar decomposition of $T_\cG$.
\end{rem}

\subsection{Restricted equivalence between frames for subspaces}

Next, we introduce the main concept of this section and develop some of its properties.

\begin{fed}\label{def frames J equiv}
Let $\cF=\{f_n\}_{n\geq 1}$ and $\cG=\{g_n\}_{n\geq 1}$ be frames for  closed subspaces $\cS$ and $\cT$ of $\cH$, respectively. If $\cJ\subset \cB(\cH)$ is an operator ideal, then we introduce the following notions: 
\ben
\item $\cF$ and $\cG$ are \textit{$\cJ$-equivalent} if there exists $G\in\cG \ell_\cJ$ such that $Gg_n=f_n$, for all $n\geq 1$. 
\item $\cF$ and $\cG$ are \textit{$\cJ$-unitarily equivalent} if there exists $U\in\cU_\cJ$ such that $Ug_n=f_n$, for all $n\geq 1$. 
\een
\end{fed}

\begin{rem}
Notice $\cJ$-equivalence and $\cJ$-unitary equivalence are equivalence relations on the set of frames for closed subspaces of $\cH$, and $\cJ$-unitary equivalence refines $\cJ$-equivalence. Moreover, given a frame $\cF=\{f_n\}_{n\geq 1}$ for the closed subspace $\cS$, then observe that the equivalence class of $\cF$ with respect to $\cJ$-equivalence can be described as
$$
\{ \cG  : \text{$\cG$ is a frame for a closed subspace of $\cH$ and } T_\cG\in \cL\cO_\cJ(T_\cF) \}
\,.$$
Similarly, we can describe the equivalence class of $\cF$ with respect to $\cJ$-unitary equivalence as
$$
\{ \cG : \text{$\cG$ is a frame for a closed subspace of $\cH$ and } T_\cG\in \cL\cU_\cJ(T_\cF) \}
\,.$$
Hence, we can apply our results of Sections \ref{sec3} and \ref{sec4} for the description of these equivalence relations among frames for closed subspaces of $\cH$.
\end{rem}

\begin{rem}\label{rem comparac} At this point it is instructive to compare the restricted equivalence relations defined above with previous notions in the literature.

\medskip

\noi1. In \cite{RB99} R. Balan considered $\cJ$-equivalence and $\cJ$-unitary equivalence between frames for $\cH$, in the particular case when $\cJ=\cB(\cH)$. For a proper ideal $\cJ$, the equivalence relations considered in Definition \ref{def frames J equiv} are strict refinements of those considered in \cite{RB99}.
Notice that our approach not only considers general operator ideals $\cJ$, but it also applies in the context of frames for (possibly proper) subspaces of the Hilbert space $\cH$. 

\medskip

\noi
2. Let $\cE$ denote the set of epimorphisms (bounded surjective operators) of a Hilbert space $\cH$.
In \cite{CPS} Corach, Pacheco and Stojanoff considered the right action of invertible operators 
on $\mathcal E$, given by 
$$
(G,T)\mapsto TG^{-1} \peso{for} (G,T)\in\cG\ell(\cH) \times \cE\,.
$$
Since  $\cF\mapsto T_\cF$ is a bijective correspondence between the set of frames $\cF=\{f_n\}_{n\geq 1}$ for $\cH$ and elements in $\cE$, then the action considered in \cite{CPS} induces an action of $\cG\ell(\cH)$ of the set of frames, given by 
$$
(G,\cF=\{f_n\}_{n\geq 1})\mapsto \{T_\cF(G^{-1}e_n)\}_{n\geq 1} \peso{for} G\in\cG\ell(\cH)\,.
$$
Here $\{ e_n\}_{n \geq 1}$ is our fixed orthonormal basis (see Definition \ref{defi basicas frames}). Notice that this action can be considered as a re-parametrization of the space of coefficients of the frame; in this sense, it is quite different from the action considered in \cite{RB99}. Indeed, since $T_\cF$ is an epimorphism, the roles of $T_\cF$ and $T_\cF^*$ are not symmetric. Notice that our approach is more general  in the sense that we deal with closed range operators. Moreover, since the class of 
closed range operators and $\cG\ell_\cJ$ are closed under adjoints, then our approach developed in Sections \ref{sec3} and \ref{sec4} 
allows us to deal with an action similar to that considered in \cite{CPS} for general groups $\cG\ell_\cJ$. Indeed, 
for $A\in\cC\cR$ if we let 
$$
\cR \cO_\cJ(A)=\{AG^{-1}\ : \ G\in\cG\ell_\cJ\},
$$
then
$$
\cR \cO_\cJ(A) = \left(\cL \cO_\cJ(A^*)\right)^*,
$$
where $\cN^*=\{A^* \ : \ A\in\cN\}\subset \cB(\cH)$ for any set $\cN\subset \cB(\cH)$. From the previous identities and the results in Section \ref{sec3} it is possible to derive several properties of the orbits $\cR \cO_\cJ(A)$, for an arbitrary $A\in\cC\cR$. Similar remarks apply to the right action of $\cU_\cJ$ on $\cC\cR$ and on the set of frames for closed subspaces of $\cH$.

\medskip

\noi 3. In \cite{FPT02} Frank, Paulsen and Tiballi consider another {\it local} notion of equivalence between frames for subspaces. Indeed, given two frames $\cF=\{f_n\}_{n\geq 1}$ and $\cG=\{g_n\}_{n\geq 1}$ for the closed subspaces $\cS,\,\cT\subset \cH$ respectively,  $\cF$ is {\it weakly similar to} $\cG$ if there exists a bounded invertible transformation $L:\cS\rightarrow \cT$ such that $L(f_n)=g_n$, for $n\geq 1$. In this case, given a sequence $\{\alpha_n\}_{n\geq 1}\in\ell^2(\N)$ then $\sum_{n\geq 1}\alpha_n f_n=0$ if and only if $L(\sum_{n\geq 1}\alpha_n f_n)=\sum_{n\geq 1}\alpha_n g_n =0$ and hence $N(T_\cF)=N(T_\cG)$ in this case. 
Moreover, as a consequence of \cite{RB99} we see that $\cF$ is weakly similar to $\cG$ if and only if 
$N(T_\cF)=N(T_\cG)$. The authors considered this condition in their study of symmetric approximations of frames (see item 1. in Examples \ref{exa ejems1} below).
\end{rem}

The next result describes some characterizations of the $\cJ$-equivalence between frames. We abbreviate $\cF - \cG=\{  f_n - g_n \}_{n \geq 1}$. 

\begin{teo}\label{aplic otro arg1}
Let $\cF=\{f_n\}_{n\geq 1}$ and $\cG=\{g_n\}_{n\geq 1}$ be frames for  closed subspaces $\cS$ and $\cT$ of $\cH$, respectively. If $\cJ$ is a proper operator ideal, then the following statements are equivalent:
 \ben
\item[i)] $\cF$ and $\cG$ are $\cJ$-equivalent;
\item[ii)] The synthesis operator $T_{\cF-\cG}=T_\cF-T_\cG\in\cJ$ and $N(T_\cF)=N(T_\cG)$;
\item[iii)] $\cF$ and $\cG$ are weakly similar and the synthesis operator $T_{\cF-\cG}=T_\cF-T_\cG\in\cJ$;
\item[iv)] The canonical duals $\cF^\#$ and $\cG^\#$ are $\cJ$-equivalent.

\een
In this case, $P_{\cS}-P_{\cT}\in\cJ$ and $[P_{\cS}:P_{\cT}]=0$.

\pausa
Moreover, if we assume further that $\cJ$ is arithmetic mean closed, then the conditions above are also equivalent to the following:
\begin{enumerate}
\item[v)] The associated Parseval frames $\tilde \cF$ and $\tilde \cG$ are $\cJ$-equivalent and $S_\cF-S_\cG\in\cJ$. 
\end{enumerate}

\end{teo}
\begin{proof} 
$i) \to ii)$. Notice that $T_{\cF-\cG}=T_\cF-T_\cG$. 
Assume that $G\in\cG\ell_\cJ$ is such that $Gf_n=g_n$ for $n \geq 1$; then, it is clear that $G\,T_\cF=T_\cG$. Hence, $N(T_\cF)=N(T_\cG)$ and $T_{\cF-\cG}=T_\cF-T_\cG=(I-G)\,T_\cF\in\cJ$. 

\medskip

\noi $ii) \leftrightarrow iii)$. This follows from item 3. in Remark \ref{rem comparac}.

\medskip

\noi $ii) \to iv)$.  If $T_{\cF-\cG}=T_\cF-T_\cG\in\cJ$, then we get that $T_\cF T_\cF^*-T_\cG T_\cG^*\in\cJ$. By 
Lemma \ref{ab y moore penrose} we see that $(T_\cF T_\cF^*)^\dagger -(T_\cG T_\cG ^*)^\dagger\in\cJ$; then,
$$T_{\cF^\#}-T_{\cG^\#}= (T_\cF T_\cF^*)^\dagger T_\cF-(T_\cG T_\cG ^*)^\dagger T_\cG =(T_\cF T_\cF^*)^\dagger (T_\cF-T_\cG)- ((T_\cG T_\cG ^*)^\dagger - (T_\cF T_\cF ^*)^\dagger ) T_\cG\in\cJ\,.$$
Since $N(T_{\cF^\#})=N(T_\cF)=N(T_\cG)=N(T_{\cG^\#})$, then by Theorem \ref{para frames J equiv}, we see that $\cF^\#$ and $\cG^\#$ are $\cJ$-equivalent.

\medskip

\noi $iv) \to i)$. Notice that $T_{\cF^\#}=(T_\cF T_\cF^*)^\dagger T_\cF=(T_\cF^*)^\dagger =(T_\cF^\dagger )^*$. Hence, 
$T_{\cF^\#}-T_{\cG^\#}=(T_\cF^\dagger- T_\cG^\dagger)^*\in\cJ$ and then $T_\cF-T_\cG\in\cJ$, by Lemma \ref{ab y moore penrose}. Using the implication $i) \to ii)$ (that we have already proved) we get that $N(T_\cF)=N(T_{\cF^\#})=N(T_{\cG^\#})=N(T_\cG)$. Then, by Theorem \ref{para frames J equiv} we conclude that there exists $G\in \cG\ell(\cH)_\cJ$ such that $G\,T_\cF=T_\cG$. If we evaluate the previous (operator) identity in the elements of the orthonormal basis $\{e_n\}_{n \geq 1}$, then we get that  $g_n=T_\cG\,e_n=GT_\cF\, e_n=Gf_n$, for $n \geq 1$.

\medskip

Assume further that $\cJ$ is an arithmetic mean closed proper operator ideal. 
By Remark \ref{rem defi pars asoc} we get that 
$T_{\tilde \cF}=V_\cF$ and $T_{\tilde \cG}=V_\cG$, where
$T_\cF=V_\cF\,|T_\cF|$ and $T_\cG=V_\cG\,|T_\cG|$  are the polar decompositions of 
$T_\cF$ and $T_\cG$, respectively.
Theorem \ref{para frames J equiv} now shows that $\tilde \cF$ and $\tilde \cG$ are $\cJ$-equivalent if and only if $N(T_\cF)=N(V_\cF)=N(V_\cG)=N(T_\cG)$ and $V_\cF-V_\cG\in\cJ$.

\medskip

\noi $i) \to v)$. If $\cF$ and $\cG$ are $\cJ$-equivalent, then in particular $T_\cF-T_\cG\in\cJ$ and $N(T_\cF)=N(T_\cG)$. By Proposition \ref{polar dec y a menos b} we get that $V_\cF-V_\cG\in\cJ$; by the previous paragraph we see that $\tilde \cF$ and $\tilde \cG$ are $\cJ$-equivalent. Moreover, $S_\cF-S_\cG=T_\cF T_\cF^*-T_\cG T_\cG^*=T_\cF(T_\cF^*-T_\cG^*)+ (T_\cF - T_\cG ) T_\cG^*\in\cJ$.

\medskip

\noi $v) \to i)$. Conversely, if  
$\tilde \cF$ and $\tilde \cG$ are $\cJ$-equivalent and $S_\cF-S_\cG\in\cJ$ then, as mentioned before, $V_\cF-V_\cG\in\cJ$ and $V_\cF^*V_\cF=P_{N(T_\cF)}=P_{N(T_\cG)}=V_\cG^*V_\cG$. On the other hand, since $R(|T_\cF^*|)= N(T_\cF)^\perp=N(T_\cG)^\perp=R(|T_\cG^*|)$ and $|T_\cF^*|^2 - |T_\cG^*|^2=S_\cF-S_\cG\in\cJ$ then, by Lemma \ref{lem sobre raices en ideales}, we get that $|T_\cF^*|- |T_\cG^*|\in\cJ$.
Hence, $T_\cF-T_\cG=|T_\cF^*| V_\cF-|T_\cG^*| V_\cG=|T_\cF^*| (V_\cF-V_\cG)+ (|T_\cF^*| -|T_\cG^*|) V_\cG\in\cJ$.
The previous facts together with Theorem \ref{para frames J equiv} now show that $\cF$ and $\cG$ are $\cJ$-equivalent.
\end{proof}

\begin{exas}\label{exa ejems1}
Let $\cF=\{f_n\}_{n\geq 1}$ and $\cG=\{g_n\}_{n\geq 1}$ be frames for the closed subspaces $\cS$ and $\cT$ of $\cH$, respectively. Assume further that $N(T_\cF)=N(T_\cG)$ (or equivalently, that $\cF$ and $\cG$ are weakly similar).
 Then, 
\ben
\item $\cF$ and $\cG$ are $\fS_2$-equivalent if and only if 
$$ \|T_\cF-T_\cG\|_2^2=\sum_{n\geq 1} \|(T_\cF-T_\cG)(e_n)\|^2=
\sum_{n\geq 1} \|f_n-g_n\|^2<\infty\,.$$ This last conditions is also referred to as $\cF$ and $\cG$ being quadratically close (see \cite{FPT02} and the references therein).
\item Recall that $\cK$ denotes the ideal of compact operators. We have that $\cF$ and $\cG$ are $\cK$-equivalent if and only if for every $\epsilon>0$ there exists $m_0\geq 1$ such that if $m\geq m_0$,  then
$$   \|\sum_{n\geq 1} c_n\, (f_{n+m}-g_{n+m})\|\leq \varepsilon\,\|\{c_n\}_{n \geq 1}\|_{\ell^2(\N)}\,.
$$ 
\een
\end{exas}

\begin{rem}
Let $\cF=\{f_n\}_{n\geq 1}$ be a frame for the closed subspace $\cS$ and let $\cJ$ be a proper operator ideal. 
 Notice that $\cJ$-equivalent frames with $\cF$ preserve some
structural properties of $\cF$ (as opposed to frames that are equivalent in the sense of \cite{RB99,FPT02}): for example, if $\cF$ is such that 
the frame operator $S_\cF$ is a $\cJ$-perturbation of $P_{\cS}$ (notice that in this case $\cS=R(S_\cF)$), then any frame $\cG$ for a closed subspace $\cT$ that is $\cJ$-equivalent with $\cF$
 also satisfies that $S_\cG$ is a $\cJ$-perturbation of the corresponding projection $P_{\cT}$. This is a consequence of a convenient interpretation of Theorem \ref{teo de ops para framesx28} and the fact that $\cL\cO_\cJ(T_\cF)=\cL\cO_\cJ(T_\cG)$.
\end{rem}

\begin{teo}
Let $\cF=\{f_n\}_{n\geq 1}$ and $\cG=\{g_n\}_{n\geq 1}$ be frames for the closed subspaces $\cS$ and $\cT$ of $\cH$, respectively. If $\cJ$ is a proper operator ideal, then the following statements are equivalent:
 \ben
\item[i)] $\cF$ and $\cG$ are $\cJ$-unitarily equivalent;
\item[ii)] The associated Parseval frames $\tilde \cF$ and $\tilde \cG$ are $\cJ$-unitarily equivalent and $S_\cF=S_\cG$;
\item[iii)] The canonical duals $\cF^\#$ and $\cG^\#$ are $\cJ$-unitarily equivalent.
\een
In this case, $P_{\cS}-P_{\cT}\in\cJ$ and $[P_{\cS}:P_{\cT}]=0$.
\end{teo}

\begin{proof}
The proof follows from Definition \ref{def frames J equiv} and a convenient re-interpretation of Theorem \ref{teo rest orb unit}.
Indeed, if we let $A=T_\cF\in\cC\cR$ with polar decomposition $A=V_A\,|A|$ then $T_{\tilde \cF}=V_A$, $S_\cF=|A^*|^2$ and $T_{\cF^\#}=(A^*)^\dagger$, and similarly for $B=T_\cG$. Also, $\cF$ and $\cG$ are $\cJ$-unitarily equivalent if and only if there exists $U\in\cU_\cJ$ such that $B=UA$.
Thus, the equivalence of items $i)$ and $ii)$ follows from Theorem \ref{teo rest orb unit}. The equivalence of items $ii)$ and $iii)$ follows from the previous remarks and the fact that 
$ ((UA)^*)^\dagger=(A^*U^*)^\dagger = U (A^*)^\dagger$ for any $U\in\cU_\cJ$. The rest of the claims follow from  Theorem \ref{teo rest orb unit}.
\end{proof}

\begin{rem}[Distances between $\cJ$-equivalent frames]\label{rem metric frames}
Let $\cJ=\fS_\Phi$ be a  symmetrically-normed ideal. Let $\cF=\{f_n\}_{n\geq 1}$ and $\cG=\{g_n\}_{n\geq 1}$ be frames for the closed subspaces $\cS$ and $\cT$ of $\cH$, respectively. Assume further that $\cF$ and $\cG$ are $\cJ$-equivalent.
Then we can consider the following distances:
\ben
\item Since $T_\cG= G\,T_\cF$ for some $G\in\cG\ell_\cJ$, then $T_\cG-T_\cF\in\cJ$ and hence we can set
$$
d_\cJ(\cF,\cG)=\| T_\cF-T_\cG\|_{\Phi}\,.
$$ Notice $d_\cJ$ is a distance function on the orbit $\cL\cO_\cJ(T_\cF)$ of all frames that are $\cJ$-equivalent with $\cF$. In case $\cJ=\fS_2$ the ideal of Hilbert-Schmidt operators,  then the distance
$$ 
d_{\fS_2}(\cF,\cG)=\| T_\cF-T_\cG\|_{2}=\left(\sum_{n \geq 1} \|f_n-g_n\|^2 \right)^{1/2}
$$
has been used to
compute the symmetric approximations of frames (see \cite{EC19,FPT02}).
\item Motivated by \cite{RB99} we consider 
\beq \label{eq defi djl}
d_\cJ^\cL(\cF,\cG)=\inf\{ \log(1+\max\{\|G-I\|_\cJ\coma \|G^{-1}-I\|_\cJ\}) \ : \ G\in\cG\ell_\cJ\coma GT_\cF=T_\cG\}\,.
\eeq
Below we show that $d_\cJ^\cL$ is a distance function on the orbit $\cL\cO_\cJ(T_\cF)$ of all frames that are $\cJ$-equivalent with $\cF$. 
\een
\end{rem}

\begin{pro}\label{pro es una metrica}
Let $\cJ=\fS_\Phi$ be a symmetrically-normed ideal. 
Let $\cF=\{f_n\}_{n\geq 1}$ be a frame for a closed subspace $\cS$. Then, the function $d_\cJ^\cL$ defined in Eq. \eqref{eq defi djl} is a distance function on the orbit $\cL\cO_\cJ(T_\cF)$. 
\end{pro}
\begin{proof} Clearly $d_\cJ^\cL$ is non-degenerate and symmetric. Hence, we check the triangle inequality.
Without loss of generality, we can consider $\cF,\,\cG,\,\cH\in\cL\cO_\cJ(T_\cF)$. In this case, there exist $G,\,H\in\cG\ell_\cJ$ such that $T_\cG=G\,T_\cF$ and $T_\cH=H\,T_\cG$; thus, $HG\in\cG\ell_\cJ$ is such that 
$HGT_\cF=T_\cH$. We argue as in the proof of \cite[Thm. 2.7]{RB99} and consider
\begin{eqnarray*}
\|HG-I\|_{\Phi} &=&\| (H-I)(G-I)+H+G-2\|_{\Phi}\\
                & \leq & \|H-I\|_{\Phi} \|G-I\|_\Phi + \|H-I\|_{\Phi}+ \|G-I\|_\Phi \\
								&=& (\|H-I\|_{\Phi}+1) (\|G-I\|_\Phi +1)-1\,.
\end{eqnarray*} 
We point out that in the previous inequalities, we have used that the norm $\| \, \cdot \, \|_\Phi$ is sub-multiplicative. Hence, 
$$
\log (\|HG-I\|_{\Phi} +1) \leq \log (\|H-I\|_{\Phi}+1) +\log (\|G-I\|_{\Phi}+1) \,.
$$
Similarly, $\log (\|(HG)^{-1}-I\|_{\Phi}+1) \leq \log (\|H^{-1}-I\|_{\Phi}+1) +\log (\|G^{-1}-I\|_{\Phi}+1)$. The previous facts imply the triangle inequality for $d_\cJ^\cL$.
\end{proof}

There are several problems associated with the previous metrics. On the one hand, it would be interesting to obtain a closed form, or at least o more explicit variational formula for $d_\cJ^\cL(\cF,\cG)$. On the other hand, 
it is natural to consider the problems of computing the distances between $\cF$ and some subsets of $\cL\cO_\cJ(T_\cF)$ (e.g. the set of Parseval frames in $\cL\cO_\cJ(T_\cF)$ or the class of oblique duals of $\cF$ in $\cL\cO_\cJ(T_\cF)$ ) with respect to the metrics $d_\cJ$ and $d_\cJ^\cL$. The following results deal with some approximation problems associated with the distance $d_\cJ$ between frames for subspaces, induced by a symmetric norming function $\Phi$. We will deal with the case of $d_\cJ^\cL$ elsewhere.

\begin{teo}\label{teo para framesx28}
Let let $\cF=\{f_n\}_{n \geq 1}$ be a frame for a closed subspace $\cS\subset \cH$. 
Let $\cJ$ a proper operator ideal and let $\cJ_0=P_{\cS}\cJ |_{\cS}\subset \cB(\cS)$ be the compression of $\cJ$ to $\cS$. Then the following conditions are equivalent: 
\begin{itemize}
\item[i)] $S_\cF|_{\cS}\in\cG\ell _{\cJ_0}$; 
\item[ii)] $\cF$ and $\cF^\#$ are $\cJ$-equivalent;
\item [iii)] $\cF$ is $\cJ$-equivalent to some of its oblique duals;
\item [iv)] $\cF$ is $\cJ$-equivalent to its associated Parseval frame $\tilde \cF$;
\item [v)] $\cF$ is $\cJ$-equivalent to some Parseval frame for a closed subspace of $\cH$.
\end{itemize}
\end{teo}

\begin{proof}
The proof of the equivalences above follows from a convenient interpretation of the equivalences in Theorem \ref{teo de ops para framesx28}.
Indeed, if we let $A=T_\cF\in \cC\cR$ with polar decomposition $A=V_A|A|$, then $(A^*)^\dagger=T_{\cF^\#}$, $V_A=T_{\tilde \cF}$. These remarks together with  Theorem \ref{teo de ops para framesx28} show the equivalence of items $i)$, $ii)$, $iv)$ and $v)$. Clearly, $ii)$ implies $iii)$. Finally, the fact that item $iii)$ implies any of the other items above follows from the second part of Definition \ref{rem sobre duales} and Theorem \ref{teo de ops para framesx28}.
\end{proof}

\begin{teo}\label{parseval mas cercano}
Let $\cJ=\fS_\Phi$ be a symmetrically-normed ideal and let $\cF=\{f_n\}_{n \geq 1}$ be a frame for a closed subspace $\cS\subset \cH$. 
\ben
\item Assume that $\cF$ is $\cJ$-equivalent to $\cG_1$, where $\cG_1$ is a Parseval frame for a closed subspace of $\cH$. Then 
$d_\cJ(\cF,\tilde \cF)\leq d_\cJ(\cF, \cG_1)$, where $\tilde \cF$ is the Parseval frame associated to $\cF$.
\item Assume that $\cF$ is $\cJ$-equivalent to $\cG_2$, where $\cG_2$ is an oblique dual of $\cF$. Then 
$d_\cJ(\cF,\cF^\#)\leq d_\cJ(\cF, \cG_2)$, where $\cF^\#$ is the canonical dual of $\cF$.
\een
 Moreover, if $\Phi$ is a strictly Schur-convex symmetric norming function and 
$d_\cJ(\cF,\tilde \cF)= d_\cJ(\cF, \cG_1)$ then $\cG_1=\tilde \cF$. Similarly, if 
$d_\cJ(\cF,\cF^\#)= d_\cJ(\cF, \cG_2)$, then $\cG_2=\cF^\#$.
\end{teo}

\begin{proof}
The proof of the equivalences above is now an immediate consequence of Theorem \ref{dist a isometrias y inversas parciales}.
\end{proof}

\begin{rem}
Let $\cJ=\fS_\Phi$ be a  symmetrically-normed  ideal and let $\cF=\{f_n\}_{n \geq 1}$ be a frame for a closed subspace $\cS\subset \cH$. Let $\cG=\{g_n\}_{n\geq 1}$ be a Parseval frame for a closed subspace $\cT\subset \cH$  such that $\cG$ is weakly similar to $\cF$ and $T_\cF-T_\cG\in\cJ$.
Then, $N(T_\cG)=N(T_\cF)$, so by Theorem \ref{aplic otro arg1} we get that $\cG\in\cL\cO_\cJ(\cF)$. In this case Theorem \ref{parseval mas cercano} applies. If we consider the particular case where $\Phi=\|\cdot\|_2$ is the $2$-norm, which is a strictly Schur-convex symmetric norming function, we get that
$$
d_{\fS_2}(\cF,\cG)^2=\|T_\cF-T_\cG\|_2^2=\sum_{n\geq 1} \|f_n-g_n\|^2 \geq \sum_{n\geq 1} \|f_n-\tilde f_n\|^2=\|T_\cF-T_{\tilde \cF}\|_2^2=d_{\fS_2}(\cF,\tilde \cF)^2\,.
$$ Moreover, equality holds if and only if $\cG=\tilde \cF$.  
Thus, Theorem \ref{parseval mas cercano} extends \cite[Theorem 2.3]{FPT02} to symmetrically-normed ideals corresponding to strictly Schur-convex norming functions. 
Similarly, using a convenient re-interpretation of Theorem \ref{dist a isometrias y inversas parciales} in the context of frame theory our results extend \cite[Theorem 4.6]{EC19}.
\end{rem}

\subsection*{Acknowledgment}
This research was partially supported by  CONICET (PIP 2021/2023 11220200103209CO), ANPCyT (2015 1505/ 2017 0883) and FCE-UNLP (11X829).

\bigskip

{\sc (Eduardo Chiumiento)} {Departamento de  Matem\'atica \& Centro de Matem\'atica La Plata, FCE-UNLP, Calles 50 y 115, 
(1900) La Plata, Argentina  and Instituto Argentino de Matem\'atica, `Alberto P. Calder\'on', CONICET, Saavedra 15 3er. piso,
(1083) Buenos Aires, Argentina.}     
                                               
\noi e-mail: {\sf eduardo@mate.unlp.edu.ar}   

\bigskip

{\sc (Pedro Massey)} {Departamento de  Matem\'atica \& Centro de Matem\'atica La Plata, FCE-UNLP, Calles 50 y 115, 
(1900) La Plata, Argentina  and Instituto Argentino de Mate\-m\'atica, `Alberto P. Calder\'on', CONICET, Saavedra 15 3er. piso,
(1083) Buenos Aires, Argentina.}

\noi e-mail: {\sf massey@mate.unlp.edu.ar}

\end{document}